\documentclass{article}
\usepackage{currfile}
\usepackage{floatrow}
\usepackage{amsmath,amsthm,verbatim,amssymb,amsfonts,amscd,graphicx}

\usepackage{graphicx,tikz,caption,subfig}
\usepackage{mathrsfs}
\usepackage{bm}
\usepackage{float}

\usetikzlibrary{patterns}

\usetikzlibrary{arrows,shapes,positioning}
\usetikzlibrary{decorations.markings}
\tikzstyle arrowstyle=[scale=1]
\tikzstyle directed=[postaction={decorate,decoration={markings, mark=at position 0.75 with {\arrow[arrowstyle]{stealth}}}}]
\tikzstyle redirected=[postaction={decorate,decoration={markings, mark=at position 0.35 with {\arrow[arrowstyle]{stealth}}}}]
\usepackage{xcolor}
\usepackage{enumitem}
\setenumerate[1]{itemsep=0pt,partopsep=0pt,parsep=\parskip,topsep=5pt}
\setitemize[1]{itemsep=4pt,partopsep=0pt,parsep=\parskip,topsep=5pt}
\setdescription{itemsep=0pt,partopsep=0pt,parsep=\parskip,topsep=5pt}

\newtheorem{theorem}{Theorem}[section]

\newtheorem{definition}[theorem]{Definition}
\newtheorem{conjecture}[theorem]{Conjecture}

\newtheorem{lemma}[theorem]{Lemma}

\newtheorem{observation}[theorem]{Observation}
\newtheorem{claim}{Claim}[section]
\newtheorem{proposition}[theorem]{Proposition}

\newcommand{\JCTB}{{\it J. Combin. Theory Ser. B}}

\newcommand{\JGT}{{\it J. Graph Theory}}

\newcommand{\DM}{{\it Discrete Math.}}
\newcommand{\DAM}{{\it Discrete Appl. Math.}}

\newcommand{\SIAMDM}{{\it SIAM J. Discrete Math.}}

\newcommand{\SIAMC}{{\it SIAM J. Comput.}}

\newcommand{\TransAMS}{{\it Trans. Amer. Math. Soc.}}

\newcommand{\ElecCom}{{\it Electronic Journal of Combinatorics}}

\begin{document}

\title{Signed circuit $6$-covers of signed $K_4$-minor-free graphs}

\author{You Lu\thanks{Research \& Development Institute of Northwest Polytechnical University in Shenzhen, Shenzhen, Guangdong 518063, China; 
School of Mathematics and Statistics, Northwestern Polytechnical University,  Xi'an, Shaanxi 710129, China. Email: luyou@nwpu.edu.cn. Partially supported
 by NSFC (No. 12271438), Guangdong Basic and Applied Basic Research Foundation (No. 2023A1515012340) and Natural Science Foundation of Qinghai Province (No. 2022-ZJ-753).}
 ~~Rong Luo\thanks{Department of Mathematics, West Virginia University, Morgantown, WV 26505, USA.  Email:~ rluo@mail.wvu.edu. Partially supported by a grant from  Simons Foundation (No. 839830)}
 ~~Zhengke Miao\thanks{Research Institute of Mathematical Science and School of Mathematics and Statistics, Jiangsu Normal University, Xuzhou, Jiangsu 221116, China. Email:  zkmiao@jsnu.edu.cn. Partially supported
 by NSFC (No. 11971205)}
 ~~and~Cun-Quan Zhang\thanks{Department of Mathematics, West Virginia University, Morgantown, WV 26505, USA.  Email:~cqzhang@mail.wvu.edu.  Partially supported
 by an  NSF grant DMS-1700218}
}

\date{}

\maketitle

\vspace{-0.9cm}

\begin{abstract}
Bermond, Jackson and Jaeger [{\em J. Combin. Theory Ser. B} 35 (1983): 297-308] proved that every bridgeless ordinary graph $G$ has a circuit $4$-cover and Fan [{\em J. Combin. Theory Ser. B} 54 (1992): 113-122] showed that $G$ has a circuit $6$-cover which together implies that $G$ has a circuit $k$-cover for every even integer $k\ge 4$. The only left case when $k = 2$ is the well-know circuit double cover conjecture. For signed circuit $k$-cover of signed graphs, it is known that for every integer $k\leq 5$, there are infinitely many coverable signed graphs without signed circuit $k$-cover and there are signed eulerian graphs that admit nowhere-zero $2$-flow but don't admit a signed circuit $1$-cover. Fan conjectured that every coverable signed graph has a signed circuit $6$-cover. This conjecture was verified only for signed eulerian graphs and for signed graphs whose bridgeless-blocks are eulerian. In this paper, we prove that this conjecture holds for signed $K_4$-minor-free graphs. The $6$-cover is best possible for signed $K_4$-minor-free graphs. 
\end{abstract}

\vspace{-0.3cm}
\section{Introduction}

Graphs or signed graphs considered in this paper are finite and may have multiple edges or loops. For terminology and notations not defined here we follows \cite{Bondy2008, Diestel2010, KR2016, West1996}.

A \emph{signed graph} is a graph $G$ with a mapping $\sigma:E(G)\mapsto \{1,-1\}$. The mapping $\sigma$, called \emph{signature}, is sometimes implicit in the notation of a signed graph and will be specified when needed.  An edge $e$ is {\em positive} if $\sigma(e)=1$, and otherwise it is {\em negative}.   An {\it ordinary graph}  is a signed graph without negative edges and a  {\it circuit} is a connected $2$-regular graph.
A circuit in a signed graph is {\it balanced} if it has an even number of negative edges and otherwise it is {\it unbalanced}.  A \emph{signed circuit} is either a balanced circuit  or a {\em barbell}, the union of two unbalanced circuits and a (possibly trivial) path (called the {\em barbell-path}) that meets the circuits only at ends. A barbell is called a  {\em short barbell} if its barbell-path is trivial, and a {\em long barbell} otherwise. The edges of a signed circuit in a signed graph correspond to a minimal dependent set in the signed graphic matroid (see \cite{Zaslavsky1982}).

     Let $G$ be a signed graph. A family $\mathcal{F}$ of signed circuits of $G$ is called a {\it signed circuit cover} of $G$ if every edge is contained in some member of $\mathcal{F}$ and is called a {\it signed circuit $k$-cover} if each edge is contained in precisely $k$ members of $\mathcal{F}$. A signed graph  is {\em coverable} if it has  a signed circuit cover. 
   Given a coverable signed graph $G$,  the minimum length of a signed circuit cover of  $G$ is denoted by $SSC(G)$. 

Note that there is no unbalanced circuit and thus no barbell in an ordinary graph.  The circuit cover of ordinary graphs  is closely related to some mainstream areas in graph theory,
  such as, Tutte's integer flow theory \cite{Alon1985SIAM, Bermond1983JCTB, FanJGT1994, Jackson1990, Jamshy1987JCTB, Macajova2011JGT, Zhang1990JGT},
Fulkerson conjecture
\cite{FanJCTB1994}, snarks  and graph minors
\cite{Alspach1994, Jackson1994}. Thus the circuit cover of ordinary graphs has been studied extensively.

It is proved by
  Bermond,   Jackson and   Jaeger \cite{Bermond1983JCTB}
that every ordinary graph admitting a nowhere-zero $4$-flow has
$SCC(G) \leq \frac{4|E|}{3}$.
By applying Seymour's $6$-flow theorem
\cite{Seymour1981}
 or Jaeger's $8$-flow theorem
\cite{Jaeger1979},
  Alon and Tarsi
\cite{Alon1985SIAM},
  and Bermond, Jackson and Jaeger \cite{Bermond1983JCTB} proved that
every bridgeless ordinary graph $G$
 has $SCC(G)\leq \frac{25|E|}{15}$.
One of the most famous open problems in this area was proposed by
Alon and   Tarsi
\cite{Alon1985SIAM}
that
{\em every bridgeless ordinary graph $G$
 has $SCC(G)\leq \frac{21|E|}{15}$}.

Bermond, Jackson and Jaeger \cite{Bermond1983JCTB} proved that every bridgeless ordinary graph $G$ has a circuit $4$-cover and Fan \cite{fan-6cover} showed that  $G$ has a circuit $6$-cover which together implies that  $G$ has a circuit $k$-cover for every even integer $k\geq 4$. The only left case when $k=2$ is the well-known circuit double cover conjecture.

For signed graphs, M\'a\v{c}ajov\'a, Raspaud,  Rollov\'a and \v{S}koviera
\cite{Macajova2015JGT}  presented the first upper bound of $SSC(G)$. They showed that $SSC(G) \leq 11|E(G)|$ if $G$ is coverable and  the upper bound was improved by Lu et al. \cite{LCLZJCTB} to $\frac{14}{3}|E(G)|$.  More improvements were obtained  later in \cite{CF2018,  KLMRJGT, Macajova2019SIAM, WYJGT, WYSIAM}.

For $k$-cover of signed graphs, Fan \cite{Fan2018} showed that for every integer $k\leq 5$, there are infinitely many  coverable signed graphs that have no signed circuit $k$-cover and he proposed the following conjecture.

\begin{conjecture} (Fan \cite{Fan2018})
\label{6-cover conj}
Every coverable signed graph has a signed circuit $6$-cover.
\end{conjecture}

The conjecture was verified for signed eulerian graphs in \cite{BCF2019} and for signed graphs whose bridgeless-blocks are eulerian in \cite{CF2021}.

A graph $H$ is a {\em minor} of a graph $G$ if a graph isomorphic to $H$ can be obtained from $G$
by edge contractions, edge deletions and vertex deletions; if not, $G$ is {\em $H$-minor-free}.   The class of $K_4$-minor-free graphs, which includes all series-parallel graphs and outerplanar graphs, is a very important family of graph class and has been studied  by many researchers for various graph theory problems (for example see \cite{MarkJGT, DKT2010}). In this paper we study the signed circuit $k$-cover and confirm Conjecture~\ref{6-cover conj} for signed $K_4$-minor-free graphs and confirm Conjecture~\ref{6-cover conj} for this family of signed graphs.

\begin{theorem}\label{TH: main}
Every coverable signed $K_4$-minor-free graph has a signed circuit $6$-cover.
\end{theorem}

Note that every coverable signed graph $G$ with four distinct  degree $3$ vertices $x_1, x_2, y_1, y_2$ has no signed circuit $k$-cover for any $k\in [1,5]$ if $G[\{x_1, x_2\}]$ is a balanced $2$-circuit and $G[\{y_1, y_2\}]$ is an unbalanced $2$-circuit. Thus the $6$-cover in Theorem \ref{TH: main} is tight. 

Before proceeding, it is worth pointing out that the problems of flow and signed circuit cover in signed graphs are significantly more challenging than their counterparts in ordinary graphs.
For instance, while ordinary Eulerian graphs trivially allow for a nowhere-zero $2$-flow and a $1$-cover, signed Eulerian graphs can have flow values of $2$, $3$, or even $4$, as shown in \cite{Macajova2019SIAM}. Additionally, there are signed Eulerian graphs that  admit nowhere-zero $2$-flow but don't have  a $1$-cover, as demonstrated in \cite{BCF2019}. Unlike ordinary graphs, coverable signed graphs may have bridges. The intricate structures of signed graphs, including barbells, bridges, and negative loops, contribute to their heightened complexity compared to ordinary graphs. Consequently, tackling the flow and signed circuit cover problems in signed graphs requires extensive machinery and specialized approaches.

This paper is organized as follows. In Section \ref{Notation and terminology}, we introduce  more notations and terminology.  Some simple cases and reduction lemmas needed in the proof of Theorem \ref{TH: main} are presented in Section \ref{Properties}. In Section \ref{Proof}, we prove Theorem \ref{TH: main} by contradiction.

\vspace{-0.3cm}

\section{Preliminaries}
\label{Notation and terminology}



Let $G$ be a graph. A vertex $x$ is called a {\em cut vertex} of $G$ if $G-x$ has more components than $G$. A graph is {\em $2$-connected} if it is connected and has no cut vertex. Let $L_x$ represent a loop at $x$ and $L(G)$ be the set of all loops of $G$. Let $N_G(x)$ and $d_G(x)$ denote the {\em neighborhood} and the {\em degree} of $x$ in $G$, respectively, where each loop at $x$ contributes $2$ to $d_G(x)$. A {\em $d$-vertex} is a vertex with degree $d$. 
For two subsets $X, Y\subseteq V(G)$ (not necessarily disjoint), denote by $E_G[X,Y]$ the set of edges of $G$ with one end in $X$ and the other end in $Y$.
 A path with ends $x$ and $y$ is called an {\em $xy$-path}. A cycle of length $k$ is called a {\em $k$-cycle}.

Let $G$ be a signed graph. For an edge subset or subgraph $S$ of $G$, denote the set of all negative edges of $S$ by $E_N(S)$ and define the {\em sign} of $S$ to be $\sigma(S)=\Pi_{e\in S} \sigma(e)$.  A path $P$ in $G$ is {\em positive} if $\sigma(P)=1$, and {\em negative} otherwise.  The path $P$ is called a {\em subdivided edge} of $G$ if every internal vertex of $P$ is a $2$-vertex of $G$. The {\em suppressed graph} of $G$, denoted by $\overline{G}$, is the signed graph obtained from $G$ by replacing each maximal subdivided edge $P$ with a single edge $e$ and assigning $\sigma(e)=\sigma(P)$.

Given a signed graph $G$, {\em switching} at a vertex $x$ is the inversion of the signs of all edges incident with $x$.  A signed graph $G'$ is said to be {\em equivalent} to $G$ if $G'$ can be obtained from $G$ via a sequence of switchings  and is denoted by $G'\sim G$. Define the {\em negativeness} of $G$ by $\epsilon(G)=\min\{|E_N(G')| : G'\sim G\}$.  
A signed graph is {\em balanced} if its negativeness is $0$ and otherwise {\em unbalanced}. That is, a balanced signed graph is equivalent to an {\em all-positive} signed graph, i.e. an ordinary graph. It is easy to see that a signed graph is balanced if and only if all of its circuits are balanced.


  A {\em tadpole} at $x$ is the union of an $xy$-path $P$ and an unbalanced circuit $C$ with $V(P) \cap V(C) = \{y\}$. The vertex $x$ is called a {\em tail} and the path $P$ is called a {\em tadpole-path}.


\begin{definition}\label{def: 6-cover}
Let $\mathcal{F}$ be a family of signed subgraphs of a signed graph $G$. Let $t\in [0, 3]$ be an integer and $x, y$ be two distinct vertices of $G$.

\begin{itemize}
\item[\rm (1)] For each $e\in E(G)$, $\mathcal{F}(e)$ denotes the number of  members in $\mathcal{F}$ containing $e$.

\item[\rm (2)] For an edge subset or a subgraph $S$ of $G$, $\mathcal{F}$ is a \emph{signed subgraph $k$-cover} of $S$ if $\mathcal{F}(e)=k$ for each edge $e$ in $S$. In particular, $\mathcal{F}$ is a \emph{signed circuit $k$-cover} of $G$ if every member of $\mathcal{F}$ is a signed circuit.

\item[\rm (3)]   
A {\em $\Psi_{xy}(t)$-cover} is a signed subgraph $6$-cover that consists of $t$ positive $xy$-paths, $t$ negative $xy$-paths, $t$ tadpoles at $x$, $6-2t$ tadpoles at $y$, and some signed circuits.

\item[\rm (4)]   
Let $xy$ be an edge. A {\em $\Psi_{xy}^*(2)$-cover} is a $\Psi_{xy}(2)$-cover such that  for each $u\in \{x, y\}$, one tadpole at $u$  doesn't contain the  vertex in  $\{x, y\}\setminus \{u\}$, and the tadpole-path of the other tadpole at $u$ contains the edge $xy$.

\end{itemize} 

\end{definition}


Signed circuit cover and flows are closely related. It is known that a signed graph $G$ is coverable  if and only if it admits a nowhere-zero $k$-flow for some integer $k\geq 2$. 
Refining the results in \cite{Bouchet1983}, we have the following characterization.

\begin{proposition}\label{prop: coverable}
A connected signed graph $G$ is coverable if and only if $\epsilon(G)\neq 1$ and there is no cut edge $b$ such that $G-b$ has a balanced component.
\end{proposition}

\vspace{-0.3cm}
\section{$\Psi_{xy}(t)$-covers of two-terminal signed graphs}
\label{Properties}


For two integers $n_1\leq n_2$, let  $[n_1, n_2]$ denote the set of integers between $n_1$ and $n_2$. A {\em two-terminal signed graph} $H(x, y)$ is a  connected signed nonempty graph $H$ with two specified vertices, a {\em source terminal} $x$ and a {\em target terminal} $y$, where $x=y$ if and only if $H$ is a negative loop. For short, we abbreviate $H(x, y)$ to $H$ if the terminals are understood from the context. 

Let $H_i=H_i(x_i, y_i)$ be a two-terminal signed graph for each $i\in [1,n]$.   The {\em parallel connection} $\mathcal{P}(H_1, \dots, H_n)$ of $H_1, \dots, H_n$ is the two-terminal signed graph  obtained from $H_1\cup \dots \cup H_n$ by identifying $x_1, \dots, x_n$ into a source terminal and identifying $y_1, \dots, y_n$ into a target terminal. 
The {\em series connection} $\mathcal{S}(H_1, \dots, H_n)$ of $H_1, \dots, H_n$ is the two-terminal signed graph with source terminal $x_1$ and target terminal $y_n$ obtained from $H_1\cup \dots \cup H_n$ by identifying $y_{i-1}$ and $x_i$ for each $i\in [2,n]$.
If $G$ is a series connection of $H_1, \dots, H_n$ and $n$ is maximum with this property, then we call every $H_i$ a {\em part} of $G$. Let $\mathcal{B}(G)=\{H_1, \dots, H_n\}$ be the set of all parts of $G$. Obviously, $\mathcal{B}(G)$ can be partitioned into three subsets as follows:
 \begin{flalign*}
\mathcal{B}_0(G)&=\{H_i\in \mathcal{B}(G) : x_i=y_i\},\\
\mathcal{B}_1(G)&=\{H_i\in \mathcal{B}(G) : x_i\neq y_i, |E(H_i)|=1\},\\
\mathcal{B}_2(G)&=\{H_i\in \mathcal{B}(G) : x_i\neq y_i, |E(H_i)|\ge 2\}.
\end{flalign*}
Note that every member of $\mathcal{B}_0(G)$ is a negative loop and every member of $\mathcal{B}_1(G)$ is a positive or negative $K_2$. A series connection is shown in Fig.~\ref{FIG: 6-chain}.

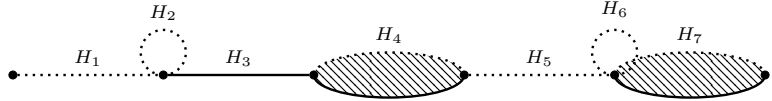
\begin{figure}[ht]

\begin{center}
\begin{tikzpicture}[scale=1]

\draw [dotted, pattern=north west lines,  line width=0.9pt] (1,0) arc (0:180: 1 and 0.3);
\draw [pattern=north west lines, line width=0.9pt] (1,0) arc (360:180: 1 and 0.3);

\draw [dotted, pattern=north west lines,  line width=0.9pt] (5,0) arc (0:180: 1 and 0.3);
\draw [pattern=north west lines, line width=0.9pt] (5,0) arc (360:180: 1 and 0.3);

\draw [dotted, line width=0.9pt] (-3,0) arc (-90:270: 0.3 and 0.3);
\draw [dotted, line width=0.9pt] (3,0) arc (-90:270: 0.3 and 0.3);



\node[above] at (-4,0){\scriptsize $H_1$};
\node[above] at (-3,0.6){\scriptsize $H_2$};
\node[above] at (-2,0){\scriptsize $H_3$};
\node[above] at (0,0.3){\scriptsize $H_4$};
\node[above] at (2,0){\scriptsize $H_5$};
\node[above] at (3,0.65){\scriptsize $H_6$};
\node[above] at (4,0.3){\scriptsize $H_7$};

{
\tikzstyle{every node}=[draw, fill=black, minimum size=3pt,inner sep=0pt,shape=circle];

\path (5,0) node (6) {};
\path (3,0) node (5) {};
\path (1,0) node (4) {};
\path (-1,0) node (3) {};
\path (-3,0) node (2) {};
\path (-5,0) node (1) {};

\draw[dotted, line width=0.9pt] (4) -- (5) ;
\draw[line width=0.9pt] (2) -- (3) ;
\draw[dotted, line width=0.9pt] (1) -- (2) ;
}
\end{tikzpicture}
\end{center}
\caption{A series connection $G$ with $\mathcal{B}_0(G)=\{H_2, H_6\}$, $\mathcal{B}_1(G)=\{H_1, H_3, H_5\}$ and $\mathcal{B}_2(G)=\{H_4, H_7\}$. Solid lines are positive; dotted lines are negative.}
\label{FIG: 6-chain}
\end{figure}


The next lemma will be applied  in the reduction.

\begin{lemma}\label{le: four paths}
Let $H_i=H_i(x_{i-1}, x_i)$ for each $i\in [1,n]$ and $G=\mathcal{S}(H_1, \dots, H_n)$ with $n=|\mathcal{B}(G)|$ and $|\mathcal{B}_2(G)|\ge 1$. Let $\theta^*=(1,1,-1,-1)$ if $|\mathcal{B}_2(G)|=1$ and $\theta^*\in \{(1,1,-1,-1), (-1,-1, -1, -1)\}$ if $|\mathcal{B}_2(G)|\ge 2$.
 If every $H_i\in \mathcal{B}_2(G)$ has a $\Psi_{x_{i-1}x_i}(2)$-cover, then $G$ has a signed subgraph $6$-cover
$$\mathcal{F}_{0}\cup 2\mathcal{B}_0(G) \cup \{P_{1}, P_{2}, P_{3}, P_{4}\}\cup \{T_{1}, T_{2}, T_{3}, T_{4}\},$$
 where 
\begin{itemize}
\item[$\rhd$] $\mathcal{F}_{0}$ is a family of signed circuits; 

\item[$\rhd$]   $P_{1}, P_{2}, P_{3}, P_{4}$ are four $x_0x_n$-paths of $G$ and $(\sigma(P_1), \sigma(P_2),\sigma(P_3), \sigma(P_4))=\theta^*$;

\item[$\rhd$] $T_{1}$ and $T_{2}$ (resp., $T_{3}$ and $T_{4}$) are two tadpoles of $G$ at $x_0$ (resp., $x_n$)  whose unbalanced circuit is in the part in $\mathcal{B}_0(G)\cup \mathcal{B}_2(G)$ with minimum (resp., maximum) subscript.
 \end{itemize}
\end{lemma}

\begin{proof} Denote $I_j=\{i : H_i\in \mathcal{B}_j(G)\}$ for each $j\in [0,2]$, and for each $i\in I_2$, let
$$
\mathcal{F}_i=\mathcal{C}_i\cup \{P_{i1}, P_{i2}, P_{i3}, P_{i4}\}\cup \{T_{i1}, T_{i2}, T_{i3}, T_{i4}\}, 
$$
be an arbitrary $\Psi_{x_{i-1}x_i}(2)$-cover of $H_i$, where $\mathcal{C}_i$ is a family of signed circuits,  $P_{i1}, P_{i2}$ (resp., $P_{i3}, P_{i4}$) are two positive (resp., negative) $x_{i-1}x_i$-paths, and $T_{i1}, T_{i2}$  (resp., $T_{i3}$,  $T_{i4}$)  are two tadpoles at $x_{i-1}$ (resp., $x_i$).

Let $\mathcal{G}_1=\left(\cup_{i\in I_2}\{P_{i1}, P_{i2}, P_{i3}, P_{i4}\}\right)\cup4\mathcal{B}_1(G)$. Note that every part in $\mathcal{B}_0(G)$ is a negative loop and every part in $\mathcal{B}_1(G)$ is a positive or negative $K_2$. Then $\mathcal{G}_1$ can be expressed as a family $\mathcal{P}$ consisting of $4$ $x_0x_n$-paths $P_1, P_2, P_3, P_4$ such that  $(\sigma(P_1), \sigma(P_2), \sigma(P_3), \sigma(P_4))=\theta^*$ and $\mathcal{G}_1(e)=\mathcal{P}(e)$ for  each $e\in E(G)$.

Let $\mathcal{G}_2=\left(\cup_{i\in I_2}\{T_{i1}, T_{i2}, T_{i3}, T_{i4}\}\right)\cup 4\mathcal{B}_0(G)\cup 2\mathcal{B}_1(G)$. For convenience, let $T_{ij}=H_i$ for each $i\in I_0$ and each $j\in [1,4]$ since $H_i$ is a tadpole at $x_{i-1}$ (=$x_i$), and $I_0\cup I_2=\{i_1, i_2, \dots, i_\ell\}$ with $0\leq i_1\leq i_2\leq \dots \leq i_\ell\leq n$. For each $j\in [1,2]$, we construct a tadpole $T_j$ at $x_0$, a tadpole $T_{j+2}$ at $x_n$, and some barbells as follows:
\begin{flalign*}
T_j &=(x_0x_1\cdots x_{i_1-1})\cup T_{i_1j},\\
T_{j+2} &=T_{i_\ell (j+2)}\cup (x_{i_\ell}\cdots x_{n-1}x_{n}),\\
B_{kj} &=T_{i_k (j+2)} \cup (x_{i_k}x_{i_k+1}\cdots x_{i_{k+1}-1})\cup T_{i_{k+1}j}, \forall\ k\in [1,\ell-1].
\end{flalign*}
Let 
$
\mathcal{T}=\{T_1, T_2, T_3, T_4\} \mbox{ and } \mathcal{C}=\cup_{k=1}^{\ell-1}\{B_{k1}, B_{k2}\}
$. Obviously, $\mathcal{G}_2(e)=(\mathcal{T}\cup \mathcal{C})(e)$ for each $e\in E(G)$. Therefore, $(\cup_{i\in I_2}\mathcal{C}_i)\cup \mathcal{C}\cup 2\mathcal{B}_0(G)\cup \mathcal{P}\cup \mathcal{T}$ is a desired signed subgraph $6$-cover of $G$.
\end{proof}

By the definition and Lemma \ref{le: four paths}, the following result is straightforward and its proof is omitted.

\begin{lemma}\label{le: block-path}
Let $H_i=H_i(x_{i-1}, x_i)$ for each $i\in [1,n]$ and $G=\mathcal{S}(H_1, \dots, H_n)$ with $n=|\mathcal{B}(G)|\ge 2$. If $\mathcal{B}_0(G)=\emptyset$, $\mathcal{B}_2(G)\neq \emptyset$ and every $H_i\in \mathcal{B}_{2}(G)$ has a $\Psi_{x_{i-1}x_i}(2)$-cover, then exactly one of the following statements holds.
               \begin{itemize}
               \item[\rm (1)] $G$ has a $\Psi_{x_0x_n}(2)$-cover in which no tadpole at $x_0$ (resp., $x_n$)  contains $x_n$ (resp., $x_0$);
               \item[\rm (2)] $\mathcal{B}_2(G)=\{H_1\}$ and every $\Psi_{x_0x_1}(2)$-cover of $H_1$ contains a tadpole at $x_1$ containing $x_0$;       
               \item[\rm (3)] $\mathcal{B}_2(G)=\{H_n\}$ and every $\Psi_{x_{n-1}x_n}(2)$-cover of $H_n$ contains a tadpole at $x_{n-1}$ containing $x_n$.               
                \end{itemize}
\end{lemma}

The next  lemma is another reduction technique in the proof of the main result.
   \begin{lemma}\label{le: 2-sum-a-b}
 Let $H_1, H_2, H_2'$ be three two-terminal signed graphs with source terminal $x$ and target terminal $y$, where $H_2$ and $H_2'$ satisfy one of the following conditions.
 \begin{itemize} 
 \item[\rm (1)] $H_2$ has a $\Psi_{xy}(t)$-cover for each $t\in [0, 3]$ in which no tadpole at $x$ contains $y$; $H_2'$ is the signed graph $D_1(x, y)$ in Fig. \ref{fig: replace}. 
 
 \item[\rm (2)] $H_2$ has a $\Psi_{xy}(2)$-cover in which no tadpole at $x$ (resp., $y$) contains $y$ (resp., $x$); $H_2'$ is the signed graph $D_2(x, y)$ in Fig. \ref{fig: replace}.
  \end{itemize} 
If $\mathcal{P}(H_1, H_2')$ has a signed circuit $6$-cover, then so does $\mathcal{P}(H_1, H_2)$.   
  \end{lemma} 


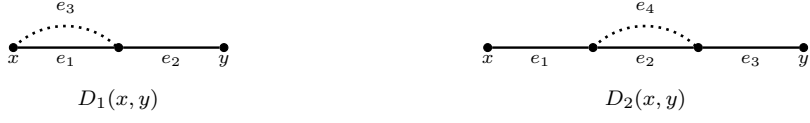
\begin{figure}[ht]
\scriptsize
\captionsetup[subfloat]{labelsep=none,format=plain,labelformat=empty}

\subfloat[$D_1(x, y)$]{
\begin{minipage}[t]{0.3\textwidth}
\begin{center}
\begin{tikzpicture}[scale=0.7]
{
\tikzstyle{every node}=[draw,fill=black, minimum size=3pt,inner sep=0pt,shape=circle];

\path (2,0) node (3) {};
\path (0,0) node (2) {};
\path (-2,0) node (1) {};

\draw[line width=0.9pt] (2) -- (3) ;
\draw[line width=0.9pt] (1) -- (2) ;
\draw[dotted, line width=1pt] (1) arc (-45:-135:-1.4cm);

}
\node[below] at (1){$x$};
\node[below] at (3){$y$};

\node[below] at (-1,0){$e_1$};
\node[below] at (-1,1){$e_3$};
\node[below] at (1,0){$e_2$};

\end{tikzpicture}
\end{center}
\end{minipage}
}
\hspace{2cm}
\subfloat[$D_2(x, y)$]{
\begin{minipage}[t]{0.3\textwidth}
\begin{center}
\begin{tikzpicture}[scale=0.7]
{
\tikzstyle{every node}=[draw,fill=black, minimum size=3pt,inner sep=0pt,shape=circle];

\path (3,0) node (4) {};
\path (1,0) node (3) {};
\path (-1,0) node (2) {};
\path (-3,0) node (1) {};

\draw[line width=0.9pt] (3) -- (4) ;
\draw[line width=0.9pt] (2) -- (3) ;
\draw[line width=0.9pt] (1) -- (2) ;
\draw[dotted, line width=1pt] (2) arc (-45:-135:-1.4cm);

}
\node[below] at (1){$x$};
\node[below] at (4){$y$};

\node[below] at (-2,0){$e_1$};
\node[below] at (2,0){$e_3$};
\node[below] at (0,0){$e_2$};
\node[below] at (0,1){$e_4$};

\end{tikzpicture}
\end{center}
\end{minipage}
}
\caption{Two two-terminal signed graphs with terminals $x$ and $y$.}
\label{fig: replace}
\end{figure}

\begin{proof}
Denote $G=\mathcal{P}(H_1, H_2)$ and $G'=\mathcal{P}(H_1, H_2')$. Let $\mathcal{F}$ be a signed circuit $6$-cover of $G'$. 

 We only prove the case when $H_2$ satisfies (1) since the augment for the  other case is very similar.
 
As shown in Fig.~\ref{fig: replace}, it follows from the structure of the signed graph $D_1(x, y)$ that, for any $C\in \mathcal{F}$, 
$$E(C)\cap E(H_2')\in \left\{\emptyset, \{e_1, e_2\}, \{e_2, e_3\}, \{e_1, e_3\}, \{e_1, e_2, e_3\}\right\}.$$
Denote by $\mathcal{F}_1$ (resp., $\mathcal{F}_2$, $\mathcal{F}_3$, $\mathcal{F}_4$) the set of signed circuits $C\in \mathcal{F}$ with 
$E(C)\cap E(H_2')=\{e_1, e_2\}$ (resp., $=\{e_3, e_2\}$, $=\{e_1, e_3\}$, $=\{e_1, e_2, e_3\}$). Since $\mathcal{F}(e_1)=\mathcal{F}(e_2)=\mathcal{F}(e_3)=6$,  
$$
|\mathcal{F}_1|+|\mathcal{F}_3|+|\mathcal{F}_4|=|\mathcal{F}_1|+|\mathcal{F}_2|+|\mathcal{F}_4|=|\mathcal{F}_2|+|\mathcal{F}_3|+|\mathcal{F}_4|=6.
$$
Thus there is an integer $t\in [0, 3]$ such that $|\mathcal{F}_1|=|\mathcal{F}_2|=|\mathcal{F}_3|=t$ and $|\mathcal{F}_4|=6-2t$. Let $\mathcal{F}_i=\{C_{i1}, \dots, C_{it}\}$ for each $i\in [1,3]$ and $\mathcal{F}_4=\{C_{41}, \dots, C_{4(6-2t)}\}$. On the other hand, by assumption, $H_2$ has a $\Psi_{xy}(t)$-cover  
$$
\mathcal{C}_0\cup \{P_{11}, \dots, P_{1t}\}\cup \{P_{21}, \dots, P_{2t}\}\cup \{P_{31}, \dots,P_{3t}\}\cup \{P_{41}, \dots, P_{4(6-2t)}\},
$$
where  no tadpole at $x$  contains $y$, $\mathcal{C}_0$ is a family of signed circuits,  each $P_{1j}$ (resp., $P_{2j}$) is a positive (resp., negative) $xy$-path, and each $P_{3j}$ (resp., $P_{4j}$) is a tadpole at $x$ (resp., $y$).
One can easily  check  that the family 
$$
\mathcal{C}_0\cup \left(\mathcal{F}\setminus (\cup_{i=1}^4\mathcal{F}_i)\right)\cup  \left(\cup_{i=1}^4\cup_{j=1}^{|\mathcal{F}_i|}\{(C_{ij}-E(H_2'))\cup P_{ij}\}\right)$$
is a signed circuit $6$-cover of $G=\mathcal{P}(H_1, H_2)$.
\end{proof}


Next we introduce  six small signed graphs  $R_i$s in Fig.~\ref{fig: six graphs}. to which some subgraphs of a smallest counterexample is reduced.
 The following observation is straightforward. 
\begin{figure}[htb]
\scriptsize
\captionsetup[subfloat]{labelsep=none,format=plain,labelformat=empty}

\subfloat[$R_0$]{
\begin{minipage}[t]{0.15\textwidth}
\begin{center}
\begin{tikzpicture}[scale=0.7]
{
\tikzstyle{every node}=[draw,fill=black, minimum size=3pt,inner sep=0pt,shape=circle];

\path (225:1.3cm) node (2) {};
\path (315:1.3cm) node (3) {};

\draw[line width=0.9pt] (2) -- (3) ;
\draw[dotted, line width=1pt] (2) arc (-45:-135:-1.3cm);

}

\node[below] at (2){$x$};
\node[below] at (3){$y$};

\end{tikzpicture}
\end{center}
\end{minipage}
}
\subfloat[$R_1$]{
\begin{minipage}[t]{0.15\textwidth}
\begin{center}
\begin{tikzpicture}[scale=0.7]
{
\tikzstyle{every node}=[draw,fill=black, minimum size=3pt,inner sep=0pt,shape=circle];

\path (90:0.7cm) node (1) {};
\path (225:1.3cm) node (2) {};
\path (315:1.3cm) node (3) {};

\draw[line width=0.9pt] (1) -- (2)(2) -- (3) (3) -- (1);

\draw[dotted, line width=1pt] (90:0.7) arc (-90:270:0.3);
}
\node[below] at (2){$x$};
\node[below] at (3){$y$};

\end{tikzpicture}
\end{center}
\end{minipage}
}
\subfloat[$R_2$]{
\begin{minipage}[t]{0.15\textwidth}
\begin{center}
\begin{tikzpicture}[scale=0.7]
{
\tikzstyle{every node}=[draw,fill=black, minimum size=3pt,inner sep=0pt,shape=circle];

\path (135:1.3cm) node (4) {};
\path (225:1.3cm) node (3) {};\path (315:1.3cm) node (3) {};

\draw[line width=0.9pt] (4) -- (2)(2) -- (3);
\draw[line width=0.9pt] (3) -- (4);
\
\draw[dotted, line width=1pt] (135:1.3cm) arc (90:0:1.8385);
}

\node[below] at (2){$x$};
\node[below] at (3){$y$};

\end{tikzpicture}
\end{center}
\end{minipage}
}
\subfloat[$R_3$]{
\begin{minipage}[t]{0.15\textwidth}
\begin{center}
\begin{tikzpicture}[scale=0.7]
{
\tikzstyle{every node}=[draw,fill=black, minimum size=3pt,inner sep=0pt,shape=circle];

\path (45:1.3cm) node (4) {};\path (135:1.3cm) node (1) {};
\path (225:1.3cm) node (2) {};\path (315:1.3cm) node (3) {};

\draw[line width=0.9pt] (1) -- (2)(2) -- (3) (1) -- (4);
\draw[line width=0.9pt] (3) -- (4);
\
\draw[dotted, line width=1pt] (1) arc (120:60:1.8385);
}

\node[below] at (2){$x$};
\node[below] at (3){$y$};

\end{tikzpicture}
\end{center}
\end{minipage}
}
\subfloat[$R_4$]{
\begin{minipage}[t]{0.15\textwidth}
\begin{center}
\begin{tikzpicture}[scale=0.7]
{
\tikzstyle{every node}=[draw,fill=black, minimum size=3pt,inner sep=0pt,shape=circle];

\path (45:1.3cm) node (4) {};\path (135:1.3cm) node (1) {};
\path (225:1.3cm) node (2) {};\path (315:1.3cm) node (3) {};

\draw[line width=0.9pt] (1) -- (2)(2) -- (3) (2) -- (4) (1) -- (4);
\draw[line width=0.9pt] (3) -- (4);
\
\draw[dotted, line width=1pt] (1) arc (120:60:1.8385);
}

\node[below] at (2){$x$};
\node[below] at (3){$y$};

\end{tikzpicture}
\end{center}
\end{minipage}
}
\subfloat[$R_5$]{
\begin{minipage}[t]{0.15\textwidth}
\begin{center}
\begin{tikzpicture}[scale=0.7]
{
\tikzstyle{every node}=[draw,fill=black, minimum size=3pt,inner sep=0pt,shape=circle];

\path (45:1.3cm) node (4) {};
\path (180:0.91925cm) node (5) {};
\path (90:0.91925cm) node (6) {};

\path (225:1.3cm) node (2) {};\path (315:1.3cm) node (3) {};

\draw[line width=0.9pt] (5) -- (2)(2) -- (3) (2) -- (4) (6) -- (4) (5) -- (6);
\draw[line width=0.9pt] (3) -- (4);
\
\draw[dotted, line width=1pt] (180:0.91925) arc (180:90:0.91925);
}

\node[below] at (2){$x$};
\node[below] at (3){$y$};

\end{tikzpicture}
\end{center}
\end{minipage}
}
\caption{Six small signed graphs with two specified vertices $x$ and $y$.}
\label{fig: six graphs}
\end{figure}
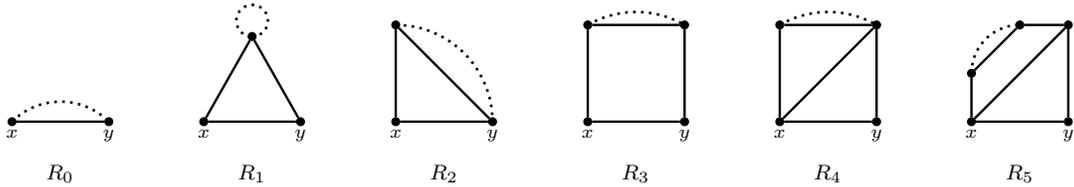 
  

\begin{observation} \label{obs: covers of R}
{\rm (1)} $R_2$ has a $\Psi_{yx}(t)$-cover for each $t\in [0, 3]$, and a $\Psi_{xy}(2)$-cover in which exactly one tadpole at $y$ doesn't contain $x$.

\noindent
{\rm (2)} $R_3$ has a $\Psi_{xy}^*(2)$-cover.

\noindent
{\rm (3)} Both $R_4$ and $R_5$ have a $\Psi_{xy}(2)$-cover $\mathcal{F}$ satisfying that $y\notin V(T_{1})\cup V(T_{2})$, $x\notin V(T_{3})$ and $xy$ is in the tadpole-path of $T_{4}$, where $\{T_{1}, T_{2}\}$ and $\{T_{3}, T_{4}\}$ are the sets of tadpoles of $\mathcal{F}$ at $x$ and $y$, respectively.
\end{observation}

By convention, for any $H=H(u, v)$, the notation $H=R_i$ (resp., $H\sim R_i$) means that $G$ is isomorphic (resp., equivalent) to $H$ and $\{u, v\}$ corresponds to $\{x, y\}$; while $H=R_i(x, y)$ (resp., $H\sim R_i(x, y)$) means that $G$ is isomorphic (resp., equivalent) to $H$ and $(u, v)$ corresponds to $(x, y)$.

\begin{lemma}\label{le: one block}
Let $H=H(x, y)$ and $G=\mathcal{P}(H\cup yz, xz)$ such that $xy\in E(H)$ and $xyzx$ is an unbalanced triangle. If $H\sim R_i(x,y)$ for some $i\in \{2, 4, 5\}$ or $H$ has a $\Psi_{xy}^*(2)$-cover, then $G$ has a $\Psi_{xz}^*(2)$-cover. 
 \end{lemma}

\begin{proof}
With  possible switching,  we assume that $\sigma(xy)=1$. If $H= \{R_2(x,y), R_4(x,y), R_5(x,y)\}$, then $G$ is a small signed graph and thus it is easy to find a $\Psi_{xz}^*(2)$-cover of $G$.  Now we assume  that $H$ has a $\Psi_{xy}^*(2)$-cover $\mathcal{F}_H$. By the definition of $\Psi_{xy}^*(2)$-cover, let  
$$\mathcal{F}_H=\mathcal{C}_0\cup \{P_1, P_2\}\cup \{Q_1, Q_2\}\cup \{T_{x1}, xy\cup T_{y2}\}\cup \{T_{y1}, yx\cup T_{x2}\},$$
where $\mathcal{C}_0$ is a family of signed circuits, $P_1, P_2$ (resp., $Q_1, Q_2$) are two positive (resp., negative) $xy$-paths, $T_{u1}, T_{u2}$  are the two tadpoles at $u$  not  containing the  vertex in $\{x, y\}\setminus\{u\}$ for each $u\in \{x, y\}$.

Let $e_0=xy$, $e_1=xz$ and $e_2=zy$. Since $xyzx$ is unbalanced and $\sigma(e_0)=1$, WLOG, assume that $\sigma(e_1)=-1$ and $\sigma(e_2)=1$.  From $G$ and $\mathcal{F}\setminus \mathcal{C}_0$, we construct an auxiliary signed graph $G'$ shown in  Fig. \ref{fig: 2-2-cover}. Observe that the family  
\begin{flalign*}
\mathcal{F}_{G'}=&\{e_3\cup e_2, e_3\cup e_2\}\cup \{e_1, e_1\} \cup \{e_5, e_1\cup e_2\cup e_6\}\cup \{e_2\cup e_6, e_1\cup e_0\cup e_4\}\\
& \cup \{e_1\cup e_2\cup e_0\cup e_5, e_1\cup e_2\cup e_4\}
\end{flalign*}
covers $\{e_1, e_2\}$ $6$ times and $E(G')\setminus \{e_1,e_2\}$ twice. 
Let $\mathcal{F}_G$ be the family obtained from $\mathcal{F}_{G'}$ by replacing two $e_3$s with $P_1, P_2$, two  $e_4$s with $Q_1, Q_2$, two  $e_5$s with $T_{x1}, T_{x2}$, two $e_6$s with $T_{y1}, T_{y2}$. One can easily check that $\mathcal{F}_G\cup \mathcal{C}_0$ is a $\Psi_{xz}^*(2)$-cover of $G$.
\end{proof}

\begin{figure}[htb]
\scriptsize
\captionsetup[subfloat]{labelsep=none,format=plain,labelformat=empty}
\begin{center}
\begin{tikzpicture}[scale=0.7]
{
\tikzstyle{every node}=[draw,fill=black, minimum size=3pt,inner sep=0pt,shape=circle];

\path (90:2cm) node (1) {};
\path (180:2cm) node (2) {};
\path (0:2cm) node (3) {};

\draw[line width=0.9pt] (1) -- (2) (3) -- (1);

\draw[dotted, line width=0.9pt] (2) -- (3);

\draw[line width=0.9pt] (1) arc (90:180:2);
\draw[dotted, line width=1pt] (90:2) arc (-90:270:0.5);
\draw[dotted, line width=1pt] (180:3) arc (180:-180:0.5);

\draw[dotted, line width=0.85] (2) .. controls (135:3cm) and (135:3cm) .. (1);
}
\node[left] at (2){$x$};
\node[right] at (3){$z$};
\node[above] at (1){$y$};

\node[below] at (0:0){$e_1$};
\node[left] at (45:1.414){$e_2$};
\node[right] at (135:1.414){$e_0$};
\node[left] at (125:1.55){$e_3$};
\node[right] at (140:3.1){$e_4$};
\node[left] at (180:3){$e_5$};
\node[left] at (98:2.8){$e_6$};
\end{tikzpicture}
\end{center}
\caption{An auxiliary signed graph $G'$.}
\label{fig: 2-2-cover}
\end{figure}
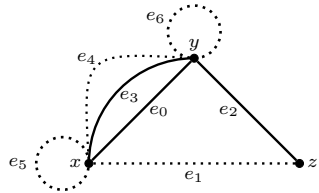


By  Observation \ref{obs: covers of R}, each member in $\{R_2, R_4, R_5\}$ has a $\Psi_{xy}(2)$-cover in which  at least one tadpole at $y$ doesn't contain $x$.  By this fact and a similar  method of the proof of Lemma \ref{le: one block},  we obtain the following lemma.

\begin{lemma}\label{le: two blocks}
Let $H_i=H_i(x, y_i)$ for each $i\in [1,2]$,  
$G=\mathcal{P}(H_1\cup y_1z, H_2\cup y_2z)$ and $G'$ be the signed graph obtained from $G$ by adding a new negative loop at $x$. 

\noindent
{\rm (1)} If $H_1=R_0$, and either $H_2\sim R_j(x, y)$ for some $j\in \{0, 2,4, 5\}$ or $H_2$ has a $\Psi_{xy_2}^*(2)$-cover, then both $G$ and $G'$ have a signed circuit $6$-cover. Moreover, $G$ has a $\Psi_{xz}(t)$-cover for  each $t\in [0, 3]$.

\noindent
{\rm (2)} If either $H_i\sim R_j(x, y)$ for some $j\in \{2,4,5\}$ or $H_i$ has a $\Psi_{xy_i}^*(2)$-cover for each $i\in [1,2]$, then both $G$ and $G'$ have a signed circuit $6$-cover. Moreover, $G$ has a $\Psi_{xz}(2)$-cover in which no tadpole at $z$ contains $x$.
 \end{lemma}

\vspace{-0.3cm}
\section{Proof of Theorem \ref{TH: main}}
  \label{Proof}
In this section, we will complete the proof of Theorem \ref{TH: main} by  contradiction. 

Let $G$ be a counterexample to Theorem \ref{TH: main} with minimum $|E(G)|$. Then $G$ is unbalanced since every coverable graph has a circuit $6$-cover (see  \cite{fan-6cover}). By the minimality, $G$ contains no $2$-vertices and can't be decomposed into two coverable signed subgraphs.
The latter implies that $G$ is connected and contains no positive loops.

\subsection{Properties of the smallest counterexample $G$}

In this subsection, we will present some properties of $G$.

For convenience, a two-terminal signed graph $H=H(x, y)$ is said to be a {\em piece} of $G$ at $\{x, y\}$ if there is another two-terminal signed graph $H'=H'(x, y)$ such that $G=\mathcal{P}(H, H')=H\cup H'$. 
       

\begin{claim}\label{le: easy observations}
 The following statements hold.

\begin{itemize}

\item[\rm (1)] No two negative loops share a common vertex.  

\item[\rm (2)] $G$ is $2$-connected.

\item[\rm (3)] Every balanced piece of $G$ is a positive or negative $K_2$.

\item[\rm (4)] If $R_1$ is a piece of $G$ at $\{x, y\}$, then $d_G(x)\ge 4$ and $d_G(y)\ge 4$.

\item[\rm (5)] $G$ contains no balanced subgraph $H=K_4-y_1y_2$, where $K_4$ is the complete graph on vertices $x, y_1, y_2, y_3$ and $x$ is a $3$-vertex of $G$. 

\item[\rm (6)] $G-L(G)$ contains no adjacent $2$-vertices.
\end{itemize}
\end{claim}
\begin{proof}
(1) Suppose to the contrary  that $e_1, e_2$ are two negative loops at a vertex. Since $C_0=e_1\cup e_2$ is a short barbell, $G-\{e_1, e_2\}$ is not coverable,  so $G-e_1$ is coverable.  By the minimality of $G$, $G-e_1$ has a signed circuit $6$-cover $\mathcal{F}$. Pick three signed circuits $C_1, C_2, C_3$ from $\mathcal{F}$ containing $e_2$.  Then the family 
$$
(\mathcal{F}\setminus \{C_1, C_2, C_3\})\cup \{C_1\bigtriangleup C_0, C_2\bigtriangleup C_0, C_2\bigtriangleup C_0\}\cup 3\{C_0\}
$$
is a signed circuit $6$-cover of $G$, a contradiction. This proves (1).

\medskip \noindent
(2) Suppose to the contrary  that there are two subgraphs $H_1, H_2$ in $G$ such that $G=H_1\cup H_2$ and $V(H_1)\cap V(H_2)=\{x\}$.  Since the minimum degree of $G$ is at least three,  $|E(H_i)|\ge 2$ for each $i\in [1,2]$.  Note that if $H_i$ is balanced, then it is coverable and thus both $H_1$ and $H_2$ are coverable, a contradiction to the minimality of  $G$. Hence  neither $H_1$ nor $H_2$ is balanced.  Therefore  for each $i \in [1,2]$,  the signed graph obtained from $H_i$ by adding a new negative loop $L_i$ at $x$ is also coverable and thus has a signed circuit $6$-cover $\mathcal{F}_i$ by the minimality of $G$ again. Let $\mathcal{C}_i=\{C_{i1}, \dots, C_{i6}\}$ be the six signed circuits in $\mathcal{F}_i$ containing $L_i$.
Then the family  
$$
(\mathcal{F}_1\setminus \mathcal{C}_1)\cup (\mathcal{F}_2\setminus \mathcal{C}_2)\cup (\cup_{j=1}^6\{(C_{1j}\setminus \{L_1\})\cup (C_{2j}\setminus \{L_2\})\})$$
is a signed circuit $6$-cover of $G$, a contradiction. This proves  (2). 

\medskip \noindent
(3) Suppose to the contrary that there are two pieces $H_1, H_2$ of $G$ at $\{x, y\}$ such that $G=H_1\cup H_2$,   $H_2$ is balanced,  and $|E(H_2)|\ge 2$. Without loss of generality, assume that $H_2$ has a positive $xy$-path. Then all $xy$-paths in $H_2$ are positive since $H_2$ is balanced.

For each $i =1,2$, let $H_i'$ be the graph obtained from $H_i$ by  adding a new positive edge $e_i$ connecting $x$ and $y$. Then $|E(H_1')|<|E(G)|$ and $H_2'$ is balanced. Moreover, both $H_1'$ and $H_2'$ are $2$-connected and $K_4$-minor-free. Obviously, $H_2'$ has a balanced circuit $6$-cover, denoted by $\mathcal{F}_2$.

 We first show that $H_1'$ is  coverable. Suppose not. Since $H_1'$ is $2$-connected,  by Proposition~\ref{prop: coverable}, there is an edge  $e$ in $H_1'$ such that $H_1'-e$ is balanced. Since $e_1$ is a positive edge in $H_1'$, every $xy$-path in $H_1-e$ is positive and thus   $G-e$ is balanced. Since $G$ is unbalanced,  it is not coverable by Proposition~\ref{prop: coverable}, a contradiction. Therefore $H_1'$ is coverable.

By the minimality of $G$, $H_1'$ has  a signed circuit $6$-cover $\mathcal{F}_1$. 
 For each $i=1, 2$, let $\mathcal{C}_i=\{C_{i1}, \cdots, C_{i6}\}$ be the six members of $\mathcal{F}_i$ containing $e_i$. Since every member of $\mathcal{C}_{2}$ is a balanced circuit, the family 
$$(\mathcal{F}_1\setminus \mathcal{C}_1)\cup (\mathcal{F}_2\setminus \mathcal{C}_2)\cup \left(\cup_{i=j}^6\{(C_{1j}\setminus\{e_1\})\cup (C_{2j}\setminus\{e_2\})\}\right)
$$ 
is a signed circuit $6$-cover of $G$, a contradiction.  This proves (3).

\medskip \noindent
 (4)  By symmetry, we only need to show that $d_G(y) \geq 4$. Suppose by contradiction that $d_G(y)=3$. Let $H=H(x,y)$ be a piece of $G$ such that $G=\mathcal{P}(H, R_1)$. As shown in Fig. \ref{fig: six graphs}, denote $C_0=xyzx$ and $R_1=C_0\cup L_z$. Clearly $G-xy$ is $2$-connected and coverable. Thus  it has a signed circuit $6$-cover $\mathcal{G}$ by the minimality of $G$.  For each $C \in \mathcal{G}$, $E(C) \cap \{L_z, xz, yz\}$ is either $\{xz, L_z\}$, or  $\{yz, L_z\}$, or  $\{xz, yz\}$, or  $\{xz, yz, L_z\}$.  Denote by $\mathcal{G}_1$ (resp., $\mathcal{G}_2$, $\mathcal{G}_3$, $\mathcal{G}_4$) the set of signed circuits $C\in \mathcal{G}$ with 
$E(C)\cap E(R_1)=\{xz, L_z\}$ (resp., $=\{yz, L_z\}$, $=\{xz, yz\}$, $=\{xz, yz, L_z\}$).  
Since $\mathcal{G}(xz)=\mathcal{G}(yz)=\mathcal{G}(L_z)=6$, we have 
$$
|\mathcal{G}_1|+|\mathcal{G}_3|+|\mathcal{G}_4|=|\mathcal{G}_2|+|\mathcal{G}_3|+|\mathcal{G}_4|=|\mathcal{G}_1|+|\mathcal{G}_2|+|\mathcal{G}_4|=6.
$$
Thus there is an integer $t\in [0, 3]$ such that  $|\mathcal{G}_1|=|\mathcal{G}_2|=|\mathcal{G}_3|=t$ and $|\mathcal{G}_4|=6-2t$. Let $\mathcal{G}_i=\{C_{i1}, \dots, C_{it}\}$ for each $i\in [1,3]$ and $\mathcal{G}_4=\{C_{41}, \dots, C_{4(6-2t)}\}$. Then the family
$$
\left\{
\begin{array}{ll}
(\mathcal{G}\setminus \{C_{41}, C_{42}, C_{43}, C_{44}\}) \cup \{C_{41}\bigtriangleup zxy, C_{42}\bigtriangleup zxy, C_{43}\bigtriangleup xyz,  C_{44}\bigtriangleup xyz\}\cup 2\{C_0\} & \mbox{if $t\in [0,1]$};\\
(\mathcal{G}\setminus \{C_{11}, C_{31}, C_{43}, C_{44}\})\cup \{C_{11}\bigtriangleup C_0, C_{31}\bigtriangleup C_0, C_{43}\bigtriangleup xyz, C_{44}\bigtriangleup xyz\}\cup 2\{C_0\} & \mbox{if $t=2$};\\
(\mathcal{G}\setminus \{C_{31}, C_{32}, C_{33}\})\cup \{C_{31}\bigtriangleup C_0, C_{32}\bigtriangleup C_0, C_{33}\bigtriangleup C_0\}\cup 3\{C_0\} & \mbox{if $t=3$}
\end{array}
\right.
$$
 is a signed circuit $6$-cover of $G$, a contradiction. This proves (4).
 
 \medskip \noindent
  (5) Suppose that such a balanced subgraph $H$ exists. Since $G$ is $K_4$-minor-free, $y_3$ is in all $y_1y_2$-paths of $G-x$ and thus $y_3$ is a cut vertex of $G-x$. Let $H_1, H_2$ be two subgraphs of $G-x$ such that $G-x=H_1\cup H_2$,  $V(H_1)\cap V(H_2)=\{y_3\}$ and $y_iy_3\in E(H_i)$ for each $i\in [1,2]$. Note that $d_{H_i}(y_i)\ge 2$. Since $G$ is $2$-connected, by (3), either $H_i$ is $2$-connected and unbalanced, or $H_i$ is the union of $y_iy_3$ and a negative loop at $y_i$. Hence $G-x$ is coverable and thus has a signed circuit $6$-cover $\mathcal{F}$ by the minimality of $G$. Pick six distinct signed circuits $C_{11}, C_{12}, C_{13}, C_{21}, C_{22}, C_{23}$ from $\mathcal{F}$ such that $y_iy_3\in E(C_{ij})$ for  each $i\in [1,2]$ and $j\in [1,3]$. Then the family
  $$
  (\mathcal{F}\setminus ( \cup_{i=1}^2\cup_{j=1}^3 \{C_{ij}\}))\cup( \cup_{i=1}^2\cup_{j=1}^3 \{C_{ij} \bigtriangleup xy_iy_3x\})\cup 3\{xy_1y_3y_2x\}
  $$
  is a signed circuit $6$-cover of $G$, a contradiction. This proves (5).

 \medskip \noindent
(6) Suppose  to the contrary that  $x, y$ are two adjacent $2$-vertices of $G-L(G)$. If $G-L_y$ is not coverable, then $E_N(G-L_y)=\{L_x\}$ and $G-\{L_x, L_y\}$ is balanced. Since $G-\{L_x, L_y\}$ is $2$-connected, it has a balanced circuit $6$-cover $\mathcal{F}$. Pick $C_1, C_2, C_3\in \mathcal{F}$ with $xy\in E(C_i)$ for $i\in [1,3]$. Then  the family
$$
(\mathcal{F}\setminus \{C_1, C_2, C_3\})\cup (\cup_{i=1}^3\{L_x\cup xy \cup L_y, L_x\cup (C_i-xy)\cup L_y\})
$$
is a signed circuit $6$-cover of $G$, a contradiction. Therefore, $G-L_y$ is coverable. Let $\mathcal{F}'$ be a signed circuit $6$-cover of  $G-L_y$ by the minimality of $G$. Similar to the proof of (4), we can extend $\mathcal{F}'$ to  a signed circuit $6$-cover of $G$, a contradiction.   This proves (6) and thus completes the proof of the claim.
 \end{proof}

\begin{claim} \label{le: 1-negative}
Let $H$ be a $2$-connected piece of $G$. If $\epsilon(H)=1$, then $H\sim R_i$ for some $i\in [0,5]$. 
   \end{claim} 
\begin{proof}  
Let $H$ and $H'$ be two pieces of $G$ at $\{x, y\}$ such that $G= \mathcal{P}(H, H')$. Without loss of generality, assume that $E_N(H)=\{e_0\}$ and the ends of $e_0$ are $z_1, z_2$ (possibly $z_1=z_2$). If $|V(H)|\leq 3$, it is obvious that $H\in \{R_0, R_1, R_2\}$. Thus we assume  $|V(H)|\ge 4$. 

We first show that $H$ is outerplanar. Since $H$ is $K_4$-minor-free, it is sufficient to prove that $H$ is $K_{2,3}$-minor-free. Suppose by contradiction that $H$ has a $K_{2,3}$-minor. Then there are two distinct vertices $u, v$ and three internally disjoint $uv$-paths $P_1, P_2, P_3$ in $H$ such that  each $|V(P_i)| \geq 3$. For each $i\in [1,3]$, let $M_i$ be the component of $H-\{u, v\}$ containing $V(P_i)\setminus \{u, v\}$, and $M_i'=H[V(M_i)\cup \{u, v\}]-uv$. Since $H$ is $K_4$-minor-free, for $\{w_1,w_2\} = \{x,y\}$ or $\{z_1,z_2\}$, there are at least two members of $\{M_1, M_2, M_3\}$ containing neither $w_1$ nor $w_2$. Therefore there is an index in $[1,3]$, say $3$,  such that $\{x, y, z_1, z_2\}\cap V(M_{3})=\emptyset$. This implies that $M_3'$ is an all-positive piece of $G$ at $\{u, v\}$. By Claim \ref{le: easy observations}-(3),  $M_{3}'$ is a positive or negative $K_2$,  a contradiction  to $|E(M_{3}')|\ge |E(P_{3})|=|V(P_i)\setminus \{u, v\}|+1\ge 2$. This proves that $H$ is outerplanar.

Let $C$ be an outer facial circuit of $H$. Similar to the above argument, we can show that  $xy\in E(C)$ and either $e_0$ is  a negative loop not at $\{x, y\}$ or there is an edge $e_1\in E(H)$ such that $e_0\cup e_1$ is an unbalanced $2$-circuit and $E(C)\cap \{e_0, e_1\}\neq \emptyset$. Without loss of generality, assume that $z_1, x, y, z_2$ appear on $C$ in the cyclic order.
 Let 
$$
P_1=z_1Cx=u_0u_1\cdots u_{p} \mbox{ and } P_2=yCz_2=v_q\cdots v_{1}v_0
$$
such that $C=z_1z_2 \cup P_1\cup xy \cup P_2$, where $u_0=z_1$, $u_{p}=x$, $v_q=y$ and $v_{0}=z_2$. 
Because $H$ is outerplanar and  the minimum degree of $G$ is at least $3$, $E(H)\setminus E(C)\subseteq E_H[V(P_1), V(P_2)]$ by Claim \ref{le: easy observations}-(3). 

If $u_0=v_0$, then $e_0$ is  a negative loop not at $\{x, y\}$. Thus $u_0u_1v_1u_0\cup e_0=R_1$. Since $|V(C)|=|V(H)|\ge 4$, either $d_G(u_1)\leq 3$ or $d_G(v_1)\leq 3$, contradicting Claim \ref{le: easy observations}-(4). Hence $u_0\neq v_0$.

Note that no two indices $i\in [0, p]$ and $j\in [0,q-2]$ with $\{v_j, v_{j+1}, v_{j+2}\}\subseteq N_{H}(u_i)$.  Otherwise $H[\{u_i, v_j, v_{j+1}, v_{j+2}\}]-e_0=K_4-v_jv_{j+2}$ and $d_G(v_{j+1})=d_H(v_{j+1})=3$, where $K_4$ is the complete graph on $\{u_i, v_j, v_{j+1}, v_{j+2}\}$, contradicting Claim \ref{le: easy observations}-(5). By the symmetry of $P_1$ and $P_2$, it follows that $p\ge 1$, $q\ge 1$ and 
\begin{equation}\label{eq: degree}
\left\{\begin{array}{ll} 
d_H(w)=3 & \mbox{if $w\in \{u_p, v_q\}$;}\\
d_H(w)\leq 4 & \mbox{if $w\in V(H)\setminus \{u_p, v_q\}$.}
 \end{array} 
\right. 
\end{equation}

Let  $H^*=H[\{u_0, u_1, v_0, v_1\}]$. According to $d_H(u_0)$ and $d_H(v_0)$, we distinguish the following two cases.  

\medskip \noindent
Case 1. $d_H(u_0)=4$ or $d_H(v_0)=4$.

By the symmetry of $P_1$ and $P_2$, assume that $d_H(v_0)=4$. Then $d_H(u_0)=3$, $u_1v_0\in E(H)$ and $H[\{u_0, u_1, v_0\}]=R_1$. If $p=1$, then $H=H^*=R_4$ by Eq. (\ref{eq: degree}).  Now we assume  $p\ge 2$. 

If $d_H(u_1)=3$, then $u_2u_1\cup H[\{u_0, u_1, v_0\}]$ is a piece of $G$ at $\{u_2, v_0\}$ and has a $\Psi_{v_0u_2}(t)$-cover for each $t\in [0,3]$ in which  no  tadpole at $v_0$ contains  $u_2$. By Lemma \ref{le: 2-sum-a-b}-(1) and the minimality of $G$, $G$ has a signed circuit $6$-cover, a contradiction.

If $d_H(u_1)=4$, then $H^*=R_4$. Let $C_0=u_1v_0v_1u_1$ and $G'=G-\{u_1v_0, u_1v_{1}\}$. Clearly, $G'$ has a signed circuit $6$-cover $\mathcal{F}'$ by the minimality of $G$. Note that $d_{G'}(u_1)=2$ and $d_{G'}(u_0)=d_{G'}(v_0)=3$. By the structure of $G'$, there are $3$ signed circuits $C_1, C_2, C_3$ in $\mathcal{F}'$ such that  each of $C_1, C_2$ contains the tadpole $e_0\cup e_1\cup v_0v_1$ but not the vertex $u_1$, and $C_3$ contains the path $u_2u_1u_0\cup e_1\cup v_0v_1$. Hence the family
$$
(\mathcal{F}'\setminus \{C_1, C_2, C_3\})\cup \{C_1\bigtriangleup C_0, C_2\bigtriangleup C_0, C_3\bigtriangleup C_0\}\cup 3\{C_0\}
$$
is a signed circuit $6$-cover of $G$, a contradiction.

\medskip \noindent
Case 2. $d_H(u_0)=d_H(v_0)=3$.

If $p=q=1$, then $H=R_3$. If $p=1$ and $q\ge 2$,  then $d_H(v_1)=3$ by Eq. (\ref{eq: degree}) and thus $q=2$ and  $H=R_5$. By the symmetry of $P_1$ and $P_2$,  we  assume that $p\ge 2$ and $q\ge 2$. Then $u_1v_1\in E(H)$.

If $d_H(u_1)=d_H(v_1)=3$, then $H^*=R_3$. By Lemma~\ref{le: block-path}-(1) $u_2u_1\cup H^*\cup v_1v_2$ has a $\Psi_{u_2v_2}(2)$-cover satisfying the condition of Lemma \ref{le: 2-sum-a-b}-(2). Together with the minimality of $G$, we can obtain a signed circuit $6$-cover of $G$, a contradiction.

Assume that either $d_H(u_1)\ge 4$ or $d_H(v_1)\ge 4$. WLOG assume $d_H(u_1)\ge 4$. Let $G'=G-\{u_1v_1, u_1v_2\}$. With a similar argument in the case when  
 $d_H(u_1)=4$ in Case 1, one can find a signed circuit $6$-cover of $G$, a contradiction.
 This completes the proof of the claim.
\end{proof}

For two distinct $x, y\in V(G)$, let $t_G(x, y)$ denote the maximum number of  pieces $H_1, \dots, H_t$ of $G$ at $\{x, y\}$ such that $G=\mathcal{P}(H_1, \dots, H_t)$.


\begin{claim}\label{le: t<4}
$t_G(x, y)\leq 3$ for any two distinct vertices $x, y\in V(G)$.
\end{claim}

\begin{proof} Suppose to the contrary that there are two distinct vertices $x, y$ such that $t=t_G(x, y)\ge 4$. Let  $H_1, \dots, H_t$  be $t$ pieces  of $G$ at $\{x, y\}$ such that  $
G=\mathcal{P}(H_1, \dots, H_t).$
Since $G$ is $K_4$-minor-free,  no  $H_i$ is  $2$-connected by the maximality of $t$.

 Without loss of generality, assume  $\epsilon(H_1)\leq \epsilon(H_2)\leq \cdots \leq \epsilon(H_t)$.  Then by Claim~\ref{le: easy observations}-(3),  $\epsilon(H_1\cup H_2)\ge 1$ and $\epsilon(H_3)\ge 1$.  Thus if $t\geq 5$, then $G$ can be decomposed into two coverable subgraphs $H_1\cup H_2\cup H_3$ and $H_4 \cup \cdots \cup H_t$, a contradiction.  Hence $t=4$.

We first consider the case when  $\epsilon(H_2)=0$.  Then $H_1\cup H_2=R_0$ by Claim \ref{le: easy observations}-(3). Since $G$ cannot   not be decomposed into two coverable subgraphs, $H_i\cup H_j$ is not coverable for some  $i\in \{1,2\}$ and $j\in \{3,4\}$. WLOG,  assume that  $H_1\cup H_3$ is not coverable and thus $\epsilon(H_1\cup H_3)=1$. By Claim \ref{le: 1-negative}, for  some $k\in [1,5]$,
\begin{equation}\label{eq: H1H2H3}
H_1\cup H_3\sim R_k\mbox{ and } H_1\cup H_2\cup H_3\sim R_k\cup e,
\end{equation}
where $e$ is a negative edge not in $R_k$ with ends $x, y$. Since  $H_1\cup H_2\cup H_3$ is coverable, $H_4$ is not coverable. Thus by Proposition \ref{prop: coverable}, there is an edge $b$ of $H_4$ such that $H_4-b$ has a balanced component $M$. Since $H_4$ is not $2$-connected, it follows from Claim \ref{le: easy observations}-(3) that $M$ is the single vertex $x$ or $y$, say $y$. Then $d_{H_4}(y)=1$. 
Let $b=yy'$. Then $H=H_1\cup H_2\cup H_3\cup b$ is a piece of $G$ at $\{x, y'\}$. By Eq.~(\ref{eq: H1H2H3}), it is easy to check that for any $t\in [0, 3]$, $H$ has a $\Psi_{xy'}(t)$-cover in which no tadpole at $x$ contains $y'$. By Lemma \ref{le: 2-sum-a-b}-(1) and  the minimality of $G$, $G$ has a signed circuit $6$-cover, a contradiction.

Now we consider the case when $\epsilon(H_2)\ge 1$. Since $\epsilon(H_2)\ge 1$,  for any $\{i, j\}\subseteq [2,4]$, $\epsilon(H_i\cup H_j)\ge \epsilon(H_i)+\epsilon(H_j)\ge 2$. Thus  $H_i\cup H_j$ is coverable. This implies that  for each $j\in [2,4]$, $H_1\cup H_j$ is not coverable and thus  by  Claim \ref{le: 1-negative}, $H_1\cup H_j\sim R_{k_j}$ for some $k_j\in [1,5]$. With some switchings, assume that $H_1$ is the positive edge $xy$. Thus for  each $j\in [2,4]$,  
$$
H_j =R_{k_j}-xy.
$$
One can check directly that $G$ has  a signed circuit $6$-cover, a contradiction. This proves the claim.
\end{proof}


\begin{claim}\label{le: H-L(H) balanced}
$H-L(H)$ is unbalanced for every $2$-connected piece $H$ of $G$.
\end{claim}

\begin{proof}
Prove by contradiction. Let $x, y$ be two distinct vertices and $H$ be a piece of $G$ at $\{x, y\}$ such that 
\begin{itemize}
\item[(i)] $H-L(H)$ is balanced;
\item[(ii)] subject to (i), $|E(H)|$ is as small as possible. 
\end{itemize}
Without loss of generality, assume that $H-L(H)$ is all-positive. Then $E_N(H)=L(H)$.  By (ii),  no member of $L(H)$ has end $x$ or $y$. Denote by $H'$ another piece of $G$ at $\{x, y\}$ such that $G=\mathcal{P}(H, H')$.

We first show  $H=R_1$. It is sufficient to prove that $H-L(H)$ is a $3$-circuit. Suppose not. Then either $H-L(H)=xz_1yz_2x$ is a $4$-circuit by Claim \ref{le: easy observations}-(6) or $\overline{H-L(H)}$ has minimum degree at least $3$.  In the  former case , $H=xz_1yz_2x \cup \{L_{z_1}, L_{z_2}\}$ and hence any signed circuit $6$-cover of $H'\cup \{L_x, L_y\}$ can be extended to a signed circuit $6$-cover of $G$, where $L_u$ is a new negative loop at $u$ for $u\in \{x, y\}$, a contradiction. In  the latter case, since $G=H\cup H'$ is $2$-connected and $K_4$-minor-free, $\overline{G-L(H)}$ contains a $2$-circuit, which corresponds to a $2$-connected piece of $H$ (and thus $G$) satisfying (i), a contradiction to (ii).  Thus $H=R_1$.

Next we show that  $H'$  has a cut-edge. Otherwise $H'$ is $2$-edge-connected. Since $G=\mathcal{P}(H, H')$ and $H$ is $2$-connected, $t_{H'}(x, y)=t_G(x, y)-t_H(x, y)\leq 1$ by Claim \ref{le: t<4}. Thus  $H'$ contains cut vertices separating $x$ from $y$.
This implies that there are $s$ $(\ge 2)$ $2$-connected subgraphs or negative loops $B_1, \dots, B_s$ such that $H'=\mathcal{S}(B_1, \dots, B_s)$ with $x\in V(B_1)$ and $y\in V(B_s)$. By Claim \ref{le: easy observations}-(3), $\epsilon(B_i)\ge 1$ for  each $i\in [1,s]$. Moreover we have $s=2$ and $\epsilon(B_1)=\epsilon(B_2)=1$  since  $G$ can't be decomposed into two coverable subgraphs.
  By Claim \ref{le: 1-negative}, $B_1\sim R_{j_1}$ and $B_2\sim R_{j_2}$ for some  $j_1, j_2\in [0, 5]$. Since $H = R_1$, one can find  a signed circuit $6$-cover of $G$, a contradiction. Thus $H'$ has a cut-edge. 
  
By the above two claims,  let $H=C_0\cup L_z$ where $C_0=xzyx$,  let $uv$ be a cut-edge of $H'$, and let  $M_1, M_2$ be two components of $H'-uv$ with $x, u\in V(M_1)$ and $y, v\in V(M_2)$. 

Let $G'=G-xy$.  Then  $G'$ is $2$-connected and coverable. By the minimality of $G$, $G'$  has a signed circuit $6$-cover. Choose a signed circuit $6$-cover $\mathcal{F}'$ of $G'$ such that the number of balanced circuits and short barbells in $\mathcal{F}'$ is as large as possible. 

To complete the proof, we will construct a signed circuit $6$-cover $\mathcal{F}$ of $G$ from $\mathcal{F}'$.

With a similar argument of the proof of Claim \ref{le: easy observations}-(4), one can show that  there is an integer $t\in [0, 3]$ and four families $\mathcal{F}_i=\{C_{i1}, \dots, C_{it_i}\}$, $i\in [1, 4]$, in $\mathcal{F}'$ such that $t_1=t_2=t_3=t$, $t_4=6-2t$ and for every $C\in \mathcal{F}_1$ (resp., $\mathcal{F}_2$, $\mathcal{F}_3$, $\mathcal{F}_4$), $E(C)\cap E(H)=\{L_z, zx\}$ (resp., $=\{L_z, yz\}$, $=\{zx, yz\}$, $=\{L_z, zx, yz\}$).

If $t\in [0,1]$,  let
 \begin{equation*}\label{eq: b=1}
\mathcal{F}=
(\mathcal{F}'\setminus \{C_{41}, C_{42}, C_{43}, C_{44}\})\cup \{C_{41}\bigtriangleup zxy, C_{42}\bigtriangleup zxy, C_{43}\bigtriangleup xyz,C_{44}\bigtriangleup xyz\}\cup 2\{C_0\}.
 \end{equation*}

If $t=3$,  let
$
\mathcal{F}=(\mathcal{F}'\setminus \{C_{31}, C_{32}, C_{33}\})\cup \{C_{31}\bigtriangleup C_0, C_{32}\bigtriangleup C_0, C_{33}\bigtriangleup C_0\}\cup 3\{C_0\}.
$

If $t=2$ and either $y\notin V(C_{11})\cap V(C_{12})$ or $x\notin V(C_{21})\cap V(C_{22})$, say $y\notin V(C_{11})$,  let 
\begin{equation*}\label{eq: b=2}
\mathcal{F}= (\mathcal{F}'\setminus \{C_{11}, C_{31}, C_{41}, C_{42}\})\cup \{C_{11}\bigtriangleup C_0, C_{31}\bigtriangleup C_0, C_{41}\bigtriangleup xyz, C_{42}\bigtriangleup xyz\}\cup 2\{C_0\}.
\end{equation*}

In each of the above case, we obtain a signed circuit $6$-cover of $G$, a contradiction.
 
Finally we consider the case that $t=2$, $y\in V(C_{11})\cap V(C_{12})$ and $x\in V(C_{21})\cap V(C_{22})$.

Then $uv\in \cap_{j=1}^2\left(E(C_{1j})\cap E(C_{2j})\cap E(C_{4j})\right)$ but $uv\notin E(C_{31})\cup E(C_{32})$. For each $j\in [1,2]$,  denote by $P_{1j}$ (resp., $T_{2j}$, $P_{4j}^1$),  the segment of $C_{1j}$ (resp., $C_{2j}$, $C_{4j}$) in $M_1$, and by $T_{1j}$ (resp., $P_{2j}$, $P_{4j}^2$) the segment of $C_{1j}$ (resp., $C_{2j}$, $C_{4j}$) in $M_2$. 
 Thus
$$
C_{1j}=L_z\cup zx\cup P_{1j}\cup uv \cup T_{1j},\  C_{2j}=L_z\cup zy\cup P_{2j}\cup vu\cup T_{2j},\  C_{4j}=L_z\cup zx \cup P_{4j}^1\cup uv \cup P_{4j}^2\cup yz.
$$
Clearly  $P_{1j}$ and $P_{4j}^1$ are $xu$-paths,  $P_{2j}$ and $P_{4j}^2$ are  $vy$-paths, and $T_{1j}$ (resp., $T_{2j}$) is a tadpole at $v$ (resp., $u$).

Since $uv\notin E(C_{31})$ and $z$ is a cut vertex of $G'-uv$, $C_{31}$ is a barbell and $z$ is in the barbell-path of $C_{31}$. Hence there are two barbells, denoted by $C_{31}^1, C_{31}^2$,  in $C_{31}\cup L_z$ such that $\{C_{31}^1, C_{31}^2\}$ covers $C_{31}$ once and $L_z$ twice. 

If  $\sigma(P_{1j_1})\sigma(P_{2j_2})\neq \sigma(P_{41}^1)\sigma(P_{41}^2)$ for some  $j_1, j_2\in [1, 2]$, then $C_1=zx \cup P_{1j_1}\cup uv \cup P_{2j_2}\cup yz$ is a balanced circuit. Let
$$
\mathcal{F}''=(\mathcal{F}'\setminus \{C_{1j_1}, C_{2j_2}, C_{31}\})\cup \{C_1, C_{31}^1, C_{31}^2, T_{1j_1}\cup uv\cup T_{2j_2}\}.
$$
If $\sigma(P_{1j_1})\sigma(P_{2j_2})=\sigma(P_{41}^1)\sigma(P_{41}^2)$ for any $j_1, j_2\in [1, 2]$, then both $C_2=zx \cup P_{11}\cup uv \cup P_{21}\cup yz$ and $C_3=zx \cup P_{12}\cup uv \cup P_{22}\cup yz$ are unbalanced circuits. Let 
$$
\mathcal{F}''=(\mathcal{F}'\setminus \{C_{11}, C_{12}, C_{21}, C_{22}, C_{31}\})\cup \{C_2\cup L_z, C_3\cup L_z, C_{31}^1, C_{31}^2, T_{11}\cup vu\cup T_{21}, T_{12}\cup vu\cup T_{22}\}.
$$
In both cases, $\mathcal{F}''$ is a signed circuit $6$-cover of $G'$ with a larger number of balanced circuits and short barbells than $\mathcal{F}'$, a contradiction to the choice of $\mathcal{F}'$. This completes the proof of the claim.
\end{proof}

\begin{claim}\label{le: balanced triangle}
Every balanced $3$-circuit is in a piece $H$ of $G$ with $H\sim R_i$ for some $i\in \{2, 4, 5\}$. 
\end{claim}

\begin{proof}
Let $C=xyzx$ be a balanced triangle. For any $\{u, v\}\subseteq V(C)$, 
$$
V_{uv}=\{w\in V(G)\setminus V(C) : \mbox{ there is a $uv$-path containing $w$ but not the third vertex  $V(C)$ in $G$}\}.
$$ 
Since $G$ is $2$-connected and $K_4$-minor-free, $\{V_{xy}, V_{xz}, V_{yz}\}$ is a partition of $V(G)\setminus V(C)$.  Let $G_{uv}=G[V_{uv}\cup \{u, v\}]$, where every loop at $V(C)$ belongs to exactly one of $\mathcal{G}=\{G_{xy}, G_{xz}, G_{yz}\}$. Then $G_{uv}$ is a piece of $G$ at $\{u, v\}$ and 
$$G=\mathcal{P}(\mathcal{S}(G_{xz},G_{zy}), G_{xy})=G_{xz}\cup G_{zy}\cup G_{xy}.$$
Without loss of generality, assume that $\epsilon(G_{xy})\ge \epsilon(G_{yz})\ge \epsilon(G_{xz})$. Note that, by the definition and Claim \ref{le: easy observations}-(3), every $G_{uv}\in \mathcal{G}$ is a positive edge if $\epsilon(G_{uv})=0$ and is $2$-connected if $\epsilon(G_{uv})\ge 1$.

If $\epsilon(G_{xz})\ge 1$, then $\epsilon(G_{xy})\ = \epsilon(G_{yz}) = \epsilon(G_{xz}) =1$; otherwise  $G$ can be decomposed into two coverable subgraphs $G_{xy}$ and $G_{xz}\cup G_{yz}$, a contradiction.  By Claim~\ref{le: 1-negative}, every $G_{uv}\in \mathcal{G}$ is equivalent to $R_i$ for some $i\in [0, 5]$. One can check easily that $G$ has a signed circuit $6$-cover, a contradiction.  Therefore $\epsilon(G_{xz}) = 0$.  By Claim~\ref{le: easy observations}-(3), $G_{xz}=xz$.

Note that $G_{yz} \not = yz$ otherwise $z$ is a $2$-vertex of $G$.  Thus $\epsilon(G_{yz}) \geq 1$.

If  $\epsilon(G_{yz}) \geq 2$, then  $\epsilon(G_{xy}) \geq \epsilon(G_{yz})\geq 2$. This implies that  both $G_{xy}$ and $G_{yz}$ are coverable. By the minimality of $G$, let $\mathcal{F}_1$ and $\mathcal{F}_2$ be two signed circuit $6$-covers of $G_{xy}$ and $G_{yz}$, respectively. For each  $i\in [1,2]$, pick three members $C_{i1}$, $C_{i2}$, $C_{i3}$ from $\mathcal{F}_i$ such that $xy\in E(C_{1j})$ and $yz\in E(C_{2j})$ for $j\in [1,3]$. Then 
$$\cup_{i=1}^2\left((\mathcal{F}_i\setminus \{C_{i1}, C_{i2}, C_{i3}\})\cup \{C_{i1}\bigtriangleup C, C_{i2}\bigtriangleup C, C_{i3}\bigtriangleup C\}\right)$$
is a signed circuit $6$-cover of $G$, a contradiction. Therefore $\epsilon(G_{yz})=1$.

Since $C$ is balanced, $\epsilon(G_{yz}\cup C)=\epsilon(G_{yz})=1$. By  Claims \ref{le: 1-negative} and \ref{le: H-L(H) balanced}, each  $H\in \{G_{yz}, G_{yz}\cup C\}$  is equivalent to an $R_i$ for some $i\in \{0, 2, 3,4,5\}$. Therefore  $G_{yz}\cup C\sim R_i$ for some $i\in \{2, 4, 5\}$.  This completes the proof of the claim.
\end{proof}


\begin{claim}\label{cl: h=2}
Let $B_i=B_i(x_{i-1}, x_i)$ for $i\in [1,h]$ and $H=\mathcal{S}(B_1, \dots, B_h)$ be a piece of $G$ at $\{x_0, x_h\}$ such that $h=|\mathcal{B}(H)|\ge 2$, $\epsilon(G-E(H))\ge 1$,  and every $B_i\in \mathcal{B}_{2}(H)$ has a $\Psi_{x_{i-1}x_i}(2)$-cover. 
\begin{itemize}
\item[\rm (1)] If $\mathcal{B}_0(H)=\emptyset$, then either $H\sim D_2(x, y)$ in Fig. \ref{fig: replace}, or $h=2$ and $\mathcal{B}_2(H)\in \{\{B_1\}, \{B_2\}\}$. Furthermore, when $\mathcal{B}_2(H)=\{B_1\}$, the following statements hold. 
\begin{itemize}

\item[\rm (1a)] Every $\Psi_{x_0x_1}(2)$-cover of $B_1$ has a tadpole at $x_1$ containing $x_0$;

\item[\rm (1b)] $B_1$ has no $\Psi_{x_0x_1}(t)$-cover for some $t\in \{0, 1, 3\}$;

\item[\rm (1c)] If $B_1$ has a $\Psi_{x_0x_1}^*(2)$-cover and $e=x_0x_h\in E(G-E(H))$, then either $H\cup e$ has a $\Psi^*_{x_0x_2}(2)$-cover, or $H\cup e$ is equivalent to one of $R_2(y, x)$, $R_4(x, y)$, and  $R_5(x, y)$.
\end{itemize}

\item[\rm (2)] If $h\ge 3$, $\mathcal{B}_2(H)=\{B_{k}\}$ and $\mathcal{B}_0(H)=\{B_{k+1}\}$ for some $k\in [1, h-2]$, then   $B_{k}$ has no $\Psi^*_{x_{k-1}x_{k}}(2)$-cover and $B_{k}$ is not equivalent to $R_i$ for each $i\in \{2,4,5\}$.
\end{itemize}
\end{claim}

\begin{proof}
Let $H'$ be a piece of $G$ at $\{x_0, x_2\}$ such that $G=\mathcal{P}(H, H')$. Then $E(H')=E(G-E(H))$.

(1) Assume  $H\not\sim D_2(x, y)$. Since $\mathcal{B}_0(H)=\emptyset$ and $\epsilon(H')\ge 1$, if $H$ has a $\Psi_{x_0x_h}(2)$-cover in which no tadpole at $x_0$ (resp., $x_h$) contains $x_h$ (resp., $x_0$), then $|E(H)|\ge 5$ and so it follows from  the minimality of $G$ and Lemma \ref{le: 2-sum-a-b}-(2) that $G$ has a signed circuit $6$-cover, a contradiction. Hence $H$ has no such $\Psi_{x_0x_h}(2)$-cover. By Lemma \ref{le: block-path}, $h=2$ and $\mathcal{B}_2(H)\in \{\{B_1\}, \{B_2\}\}$.  

Assume  $\mathcal{B}_1(H_1)=\{B_1\}$. Clearly (1a)  follows from Lemma \ref{le: block-path} and  (1b)  follows from Lemma \ref{le: 2-sum-a-b}-(1). 

We now prove (1c). Suppose to the contrary that $H\cup e$ has no $\Psi^*_{x_0x_2}(2)$-cover and $H\cup e$ is  not equivalent to any of $R_2(y, x)$, $R_4(x, y)$, and  $R_5(x, y)$.   Furthermore since $x_2$ is a $2$-vertex of $H\cup e$, $H\cup e\not\sim R_i$ for each $i\in \{2, 4, 5\}$. Since $B_1$ has a $\Psi^*_{x_0x_1}(2)$-cover, $x_0x_1\in E(B_1)$ by the definition. With  some switchings, assume that $x_0x_1$ is positive. By Lemma \ref{le: one block}, $C=x_0x_1x_2x_0$ is a balanced $3$-circuit. Note that $x_1$ is a $2$-vertex of $H'\cup C$. By Claim \ref{le: balanced triangle}, $(H'\cup C)(x_0, x_1)\sim R_2(y, x)$ or $R_i(x, y)$ for some $i\in \{4,5\}$, and thus $H'\sim R_i(x, y)$ for some $i\in \{0, 2, 3\}$. Since $G=\mathcal{P}(B_1\cup x_1x_2, H')$, by Lemma \ref{le: two blocks} and Observation \ref{obs: covers of R}-(2), $G$ has a signed circuit $6$-cover, a contradiction. This proves (1c).

\medskip
(2) Suppose to the contrary that either $B_{k}$ has an $\Psi^*_{x_{k-1}x_{k}}(2)$-cover or $B_{k}\sim R_i$ for some $i\in \{2, 4, 5\}$. Since  $B_{k+1}$ is a negative loop at $x_{k}$ ($=x_{k+1}$),  $B_{k}\cup B_{k+1}$ has a $\Psi_{x_{k-1}x_{k+1}}(2)$-cover in which no tadpole at $x_{k+1}$  contains $x_{k-1}$. Since $k\leq h-2$, we have $B_h\in \mathcal{B}_1(H)$. Thus  $H$ has a $\Psi_{x_0x_h}(2)$-cover in which no tadpole at $x_{0}$ (resp, $x_h$) contains $x_{h}$ (resp., $x_0$). Since $\epsilon(H')\ge 1$,  by Lemma \ref{le: 2-sum-a-b}-(2) and the minimality of $G$, $G$ has a signed circuit $6$-cover, a contradiction. 
This prove  (2) and thus completes the proof of the claim.
\end{proof}

\begin{claim}\label{le: a-b-covers and 2-2-covers}
Suppose that  $G=\mathcal{P}(H_1, H_2, H_3)$ where each $H_i=H_i(x, y)$ and  $\epsilon(H_3)\ge 1$. If $H=H_1\cup H_2\not\sim R_i$ for any $i\in \{0, 2, 4, 5\}$ and contains no negative loop at $\{x, y\}$, then the following statements hold. 
\begin{itemize}

\item[\rm (1)] If either $H_1\sim D_1(x, y)$ or $H_2\sim D_1(x, y)$, then $H$ has a $\Psi_{xy}(t)$-cover for each $t\in [0,3]$, where $D_1(x, y)$ is the two-terminal signed graph in Fig. \ref{fig: replace}.

\item[\rm (2)] If $xy\in E(H)$, then $H$ has a $\Psi_{xy}^*(2)$-cover.

\item[\rm (3)] If $xy\notin E(H)$ and neither $H_1$ nor $H_2$ is equivalent to $D_1(x, y)$, then $H$ has a $\Psi_{xy}(2)$-cover in which no tadpole at $y$ contains $x$. 
\end{itemize}
\end{claim}

\begin{proof}
Suppose that $H$ is a counterexample to the claim with minimum $|E(H)|$. Recall that $G$ is $2$-connected and contains no positive loop. By the definition, let $B_i=B_i(x_{i-1}, x_i)$, $i\in [1,s]$ such that 
$$
H_{1}(x, y)=\mathcal{S}(B_{1}, \dots, B_{h})=B_1\cup \cdots \cup B_h,\ \ H_{2}(y, x)=\mathcal{S}(B_{h+1}, \dots, B_{s})=B_{h+1}\cup \cdots \cup B_{s}
$$
and $s$ is maximum with this property, where $x=x_0=x_s\in V(B_{1})\cap V(B_{s})$ and $y=x_h\in V(B_{h})\cap V(B_{h+1})$. Then, for any $B\in \mathcal{B}_2(H_1)\cup \mathcal{B}_2(H_2)$ with terminals $\{u, v\}$, $B$ is $2$-connected by the maximality of $s$, and $B-L(B)$ is unbalanced by Claim \ref{le: H-L(H) balanced}. Furthermore, it follows from the minimality of $H$ that $B$ has either a $\Psi_{uv}(t)$-cover for each $t\in [0, 3]$, or a $\Psi_{uv}^*(2)$-cover, or a $\Psi_{uv}(2)$-cover in which no tadpole at  $v$ contains $u$, unless $B\sim R_i$ for some $i\in \{0, 2, 4, 5\}$. By this fact and Observation \ref{obs: covers of R}, $B$ has a $\Psi_{u,v}(2)$-cover.

We will find  a desired $\Psi_{xy}(2)$-cover of $H$, contradicting that $H$ is a counterexample to the claim. To do this, when $H_i$, $i\in [1,2]$, is not a single edge (that is, $|\mathcal{B}_0(H_i)| + |\mathcal{B}_2(H_i)|\ge 1$), we apply Lemma \ref{le: four paths} to construct a signed subgraph $6$-cover $\mathcal{F}_i^*$ of $H_i$ as follows:

$$\mathcal{F}_i^*=\mathcal{F}_{i0}\cup 2\mathcal{B}_0(H_i) \cup \{P_{i1}, P_{i2}, P_{i3}, P_{i4}\}\cup \{T_{i1}, T_{i2}, T_{i3}, T_{i4}\},$$
 where 
\begin{itemize}
\item[$\rhd$] $\mathcal{F}_{i0}$ is a subfamily of signed circuits of $H_i$; 

\item[$\rhd$]   $P_{i1}$ and $P_{i2}$ (resp., $P_{i3}$ and $P_{i4}$) are two positive (resp., negative) $xy$-paths of $H_i$ if $|\mathcal{B}_2(H_i)|\ge 1$, and otherwise $P_{i1}=P_{i2}=P_{i3}=P_{i4}=H_i-\mathcal{B}_0(H_i)$; 

\item[$\rhd$] $T_{i1}, T_{i2}$ (resp., $T_{i3},T_{i4}$) are two tadpoles of $H_i$ at $x$ (resp., $y$) such that the unbalanced circuit in $T_{i(2i-1)}$ (resp., $T_{i(2i)}$, $T_{i(5-2i)}$, $T_{i(6-2i)}$) is in the part in $\mathcal{B}_0(H_i)\cup \mathcal{B}_2(H_i)$ with minimum (resp., minimum, maximum, maximum) subscript.

 \end{itemize}
Note that $P_{11}\cup P_{21}$ is a circuit and every part in $\mathcal{B}_0(H_1)\cup \mathcal{B}_0(H_2)$ is a negative loop. When $\mathcal{B}_0(H_1)\cup \mathcal{B}_0(H_2)\neq \emptyset$, the signed graph $P_{11}\cup P_{21}\cup \mathcal{B}_0(H_1)\cup \mathcal{B}_0(H_2)$ has a family 
$$
\mathcal{C}_0\cup \{T_{1}', T_{2}'\}
$$
which covers $P_{11}\cup P_{21}$ once and $\mathcal{B}_0(H_1)\cup \mathcal{B}_0(H_2)$ twice, where $\mathcal{C}_0$ is a set of barbells and $T_{1}', T_{2}'$ are two tadpoles at $x$.

\medskip 
(1) WLOG, assume that $H_2=D_1(x, y)$. Then $h\ge 2$ since $H\not\sim R_2$. Let $t\in [0, 3]$.  

If $\mathcal{B}_2(H_1)=\emptyset$, then $H_1=xx_1y\cup L_{x_1}$ by Claim \ref{le: easy observations}-(6), and thus it  is easy to check that  $H=H_1\cup H_2$ has a $\Psi_{xy}(t)$-cover. 

If $h=2$ and $B_2\in \mathcal{B}_1(H_1)$, then $\mathcal{B}_0(H_1)=\emptyset$ and $\mathcal{B}_2(H_1)=\{B_1\}$. By (1a) and (1b) of Claim \ref{cl: h=2}, $B_1$ has a $\Psi_{x_0x_1}^*(2)$-cover and so    $H$ has a $\Psi_{xy}(t)$-cover by Lemma \ref{le: two blocks}-(1). 

Next assume that either $h\ge 3$ and $\mathcal{B}_2(H_1)\neq \emptyset$, or $h=2$ and $B_2\in \mathcal{B}_2(H_1)$. Then $x\notin V(T_{13})\cup V(T_{14})$. 
We construct a family $\mathcal{F}^*$ as follows. 
\begin{flalign*}
\mathcal{F}^*&=  \mathcal{F}_{10}\cup \mathcal{F}_{20}\cup \{P_{11}\cup P_{21}, T_{11}\cup T_{21}\}\\
& \cup \left\{
\begin{array}{ll}
\{P_{14}\cup P_{24}, T_{12}\cup T_{22}\}\cup \{P_{12}\cup P_{23}, P_{13}\cup P_{22}, T_{13}, T_{14}, T_{23}, T_{24}\} & \mbox{if $t=0$};\\
\{P_{14}\cup P_{24}, T_{12}\cup T_{22}\}\cup \{P_{12}, P_{23}\}\cup \{P_{13}\cup P_{22}, T_{13}, T_{14}, T_{23},  T_{24}\} & \mbox{if $t=1$};\\
\{P_{14}\cup P_{24}, T_{13}\cup T_{23}\}\cup \{P_{12}, P_{22}, P_{13}, P_{23}\}\cup \{T_{12}, T_{22}, T_{14}, T_{24}\} & \mbox{if $t=2$};\\
\{T_{12}\cup T_{22}\}\cup \{P_{12}, P_{22}, P_{24}\bigtriangleup B_s, P_{13}, P_{14}, P_{23}\}\cup \{B_s, T_{13}\cup P_{24}, T_{14}\cup P_{24}\} & \mbox{if $t=3$}.
\end{array}
\right.
\end{flalign*}
When $|\mathcal{B}_0(H_1)|=0$, let $\mathcal{F}=\mathcal{F}^*$. When $\mathcal{B}_0(H_1)=\{B_i\}$ for some $i\in [2,h-1]$,  let $\mathcal{C}_0=\emptyset$ and 
$$
\mathcal{F}=  
\left\{
\begin{array}{ll}
\left(\mathcal{F}^*\setminus \{P_{11}\cup P_{21}, P_{14}\cup P_{24}\}\right) \cup \{P_{11}\cup P_{24}\cup B_i, P_{14}\cup P_{21}\cup B_i\} & \mbox{if $t\in [0,2]$};\\
\left(\mathcal{F}^*\setminus \{P_{11}\cup P_{21}, B_s\}\right)\cup \{B_s\cup T_{1}'\} \cup \{T_{2}'\} & \mbox{if $t=3$}.
\end{array}
\right.
$$
When $|\mathcal{B}_0(H_1)|\ge 2$, let $\mathcal{F}=  \left(\mathcal{F}^*\setminus \{P_{11}\cup P_{21}\}\right) \cup \mathcal{C}_0\cup \{T_{1}'\cup T_{2}'\}.$
In each case, one can easily check that $\mathcal{F}$ is a $\Psi_{xy}(t)$-cover of $H$ by the structure of $H_2=D_1(x, y)$.

\medskip
 (2) WLOG, assume that $H_2=xy$ is positive. Then $h\ge 2$ since $H\not\sim R_0$.
 
If $\mathcal{B}_2(H_1)=\emptyset$,  then $H_1=xx_1y\cup L_{x_1}$ by Claim \ref{le: easy observations}-(6). Thus  $H=xx_1yx\cup L_{x_1}$ is a short barbell by Claim \ref{le: H-L(H) balanced} and  has a $\Psi_{xy}^*(2)$-cover. 

If $\mathcal{B}_0(H_1)=\emptyset$, then by Claim \ref{cl: h=2}-(1), either $H_1\sim D_2(x, y)$ in Fig. \ref{fig: replace} or $h=2$ and $\mathcal{B}_2(H_1)=\{B_1\}$ or $\{B_2\}$. For the former,  $H\sim R_3$ and thus has a $\Psi_{xy}^*(2)$-cover. For the latter, by the symmetry, assume that $\mathcal{B}_2(H_1)=\{B_1\}$, and then $B_1$ has a $\Psi^*_{x_0x_1}(2)$-cover by (1a) and (1b) of Claim \ref{cl: h=2}. Since $H\not\sim R_i$ for each $i\in \{2,4,5\}$, $H$ has a $\Psi_{xy}^*(2)$-cover by (1c) of Claim \ref{cl: h=2}.

Now we  assume that $\mathcal{B}_2(H_1)\neq \emptyset$ and $\mathcal{B}_0(H_1)\neq \emptyset$. Then $h\ge 3$. Let $B_{k}$ (resp.,  $B_{\ell}$) be the part in $\mathcal{B}_2(H_1)\cup \mathcal{B}_0(H_1)$ with minimum (resp., maximum) subscript.

If $V(B_{k})\cap V(B_{\ell})=\emptyset$, then by the choice of $\mathcal{F}^*_1$, $V(T_{11})\cap V(T_{13})=\emptyset$. Thus $T_{11}\cup \{xy\}\cup T_{13}$ is a barbell. Therefore, the family
$$
\mathcal{F}_{10}\cup \mathcal{C}_0\cup \{P_{12}\cup xy, T_{11}\cup xy\cup T_{13}\}\cup \{xy, xy, P_{13}, P_{14}\}\cup \{T_{1}', T_{2}', T_{12}\cup xy, T_{14}\}
$$
is a $\Psi_{xy}^*(2)$-cover of $H$.

 If $V(B_{k})\cap V(B_{\ell})\neq \emptyset$, then either $\mathcal{B}_2(H_1)=\{B_{k}, B_{k+2}\}$ and $\mathcal{B}_0(H_1)=\{B_{k+1}\}$, or $\mathcal{B}_2(H_1)\cup \mathcal{B}_0(H_1)=\{B_{k}, B_{k+1}\}$. In the former case, by the proof of Lemma \ref{le: four paths}, there are $4$ negative $x_0x_h$-paths $P_{11}', P_{12}', P_{13}', P_{14}'$ in $H_1$ such that $
\left(\mathcal{F}_1^*\setminus \{P_{11}, P_{12}, P_{13}, P_{14}\}\right)\cup \{P_{11}', P_{12}', P_{13}', P_{14}'\}$ is a signed subgraph $6$-cover of $H_1$ and hence the family
$$
\mathcal{F}_{10}\cup \{P_{11}'\cup xy\cup B_{k+1}, P_{12}'\cup xy\cup B_{k+1}\}\cup \{xy, xy, P_{13}', P_{14}'\}\cup \{T_{11}, xy\cup T_{13}, yx\cup T_{12}, T_{14}\}
$$
is a $\Psi_{xy}^*(2)$-cover of $H$. In  the latter case, assume that $\mathcal{B}_2(H_1)=\{B_{k}\}$ and $\mathcal{B}_0(H_1)=\{B_{k+1}\}$ by the symmetry. Then $k=h-2\in [1,2]$ since $H$ has no negative loop at $x_h$ and $G$ contains no $2$-vertex. By Claim \ref{cl: h=2}-(2), $B_{k}$ has no $\Psi_{x_{k-1}x_{k}}^*(2)$-cover and $B_{k}\not\sim R_i$ for each $i\in \{2, 4, 5\}$. Hence $k=1$; otherwise $B_2\cup B_1$ is a piece of $G$ at $\{x_2, x_0\}$ and thus, by 1a) and 1b) of Claim \ref{cl: h=2}, $B_1$ has a $\Psi_{x_{2}x_{1}}^*(2)$-cover, a contradiction. Since $\mathcal{B}_0(H_1)=\{B_2\}$ and $H_2\cup H_3$ is unbalanced, $G-E(B_1)$ is coverable. Hence $B_1$ is not coverable. By Claim \ref{le: 1-negative}, $B_1=R_0$ and thus $H-L_{x_1}\sim R_2(y, x)$. Since $H$ has a unique balanced $3$-circuit $C=x_0x_1x_2x_0$, by Claim \ref{le: balanced triangle}, $C\cup H_3\sim R_i$ for some $i\in \{2, 4, 5\}$. Therefore, one can easily check that $G=(H-E(C))\cup (C\cup H_3)$ has a signed circuit $6$-cover, a contradiction.

\medskip
(3) Since $xy\notin E(H)$, both $H_1$ and $H_2$ contain cut vertices by Claim \ref{le: t<4}, and so $h\ge 2$ and $s-h\ge 2$. If $|\mathcal{B}_2(H_1)|=|\mathcal{B}_2(H_2)|=0$, then $H=x_0x_1x_2x_3x_0\cup \{L_{x_1}, L_{x_3}\}$ by Claim \ref{le: easy observations}-(6) and $H-L(H)$ is unbalanced by Claim \ref{le: H-L(H) balanced}. Thus one can easily find a desired $\Psi_{xy}(2)$-cover, a contradiction. Hence $|\mathcal{B}_2(H_1)|+|\mathcal{B}_2(H_2)|\ge 1$ and, when $|\mathcal{B}_2(H_i)|=0$, we may assume that $H_i-L(H_i)$ is positive (with possible switchings).

By the construction, we can choose $\mathcal{F}_1^*$ and $\mathcal{F}_2^*$ such that $y\notin V(T_{i1})\cap V(T_{i2})$ and $x\notin V(T_{i3})\cap V(T_{i4})$ for each $i\in [1,2]$; otherwise, if either $y\in \cup_{i=1}^2(V(T_{i1})\cap V(T_{i2}))$ or $x\in \cup_{i=1}^2(V(T_{i3})\cap V(T_{i4}))$, say $y\in V(T_{11})\cap V(T_{12})$, then $(\mathcal{B}_0(H_1), \mathcal{B}_1(H_1), \mathcal{B}_2(H_1))=(\emptyset, \{B_1\}, \{B_2\})$ and for every $\Psi_{x_1x_2}(2)$-cover of $B_2$, both its tadpoles at $x_1$ contain $x_2$, contradicting that $B_2$ has a $\Psi_{x_2x_1}^*(2)$-cover by (1a) and (1b) of Claim \ref{cl: h=2}. Therefore, WLOG, assume that $y\notin V(T_{11})\cup V(T_{21})$ and  $x\notin V(T_{14})\cup V(T_{24})$.

 If $x\notin V(T_{13})$ or $x\notin V(T_{23})$, say $x\notin V(T_{13})$, since $|\mathcal{B}_2(H_1)|+|\mathcal{B}_2(H_2)|\ge 1$, the family
\begin{flalign*}
\mathcal{F} &= \mathcal{F}_{10}\cup \mathcal{F}_{20} \cup \{T_{12}\cup T_{21}, T_{13}\cup T_{23}\}\\
& \cup 
\left\{
\begin{array}{ll}
\mathcal{C}_0\cup \{T_{11}\cup T_{22}\}\cup \mathcal{P}\cup \{T_{1}', T_{2}', T_{14}, T_{24}\} & \mbox{if $\mathcal{B}_0(H_1)\cup \mathcal{B}_0(H_2)\neq \emptyset$};\\
\{P_{11}\cup P_{21}\}\cup \mathcal{P}\cup \{T_{11}, T_{22}, T_{14}, T_{24}\} & \mbox{if $\mathcal{B}_0(H_1)\cup \mathcal{B}_0(H_2)= \emptyset$}.
\end{array}
\right.
\end{flalign*}
is a desired $\Psi_{xy}(2)$-cover, where $\mathcal{P}=\{P_{12}, P_{13}, P_{22}, P_{23}\}\cup \{P_{14}\cup P_{24}\}$ if $|\mathcal{B}_2(H_1)|\ge 1$ and $|\mathcal{B}_2(H_2)|\ge 1$, and $\mathcal{P}=\{P_{13}, P_{14}, P_{23}, P_{24}\}\cup \{P_{12}\cup P_{22}\}$ otherwise. 

If $x\in V(T_{13})\cap V(T_{23})$, then for each $i\in [1,2]$, $(\mathcal{B}_0(H_i), \mathcal{B}_1(H_i), \mathcal{B}_2(H_i))=(\emptyset, \{B_{i+1}\}, \{B_{3i-2}\})$, and both $B_1$ and $B_4$ have $\Psi_{x_{j-1}x_j}^*(2)$-covers by (1a) and (1b) of Claim \ref{cl: h=2}. Therefore $H$ has a desired $\Psi_{xy}(2)$-cover by Claim \ref{le: two blocks}-(2).   This completes the proof of the claim. \end{proof}

\subsection{The final step}

 Since $\overline{G-L(G)}$ is $2$-connected, loopless, $K_4$-minor-free, and of minimum degree at least $3$,  it contains a $2$-circuit, denoted by $C_1=x_0x_1x_0$. Let $C_2$ be the circuit of $G-L(G)$ corresponding to $C_1$ and let 
$$B_1=C_2\cup \{L_z\in L(G) : z\in V(C_2)\setminus \{x_0, x_1\}\}.$$
Obviously, $B_1$ is a $2$-connected piece of $G$ at $\{x_0, x_1\}$. 
By Claims \ref{le: H-L(H) balanced} and  \ref{le: easy observations}-(6), $C_2=B_1-L(B_1)$ is an unbalanced circuit of length $2$ or $3$ or $4$, denoted by  $x_0x_1x_0$ or $x_0zx_1x_0$ or  $x_0z_1x_1z_2x_0$ depending on its length. Hence $B_1=x_0x_1x_0$ or $B_1=x_0zx_1x_0\cup L_{z}$ or $B_1=x_0z_1x_1z_2x_0\cup \{L_{z_1}, L_{z_2}\}$. 
In each case, $B_1$ has a $\Psi_{x_0x_1}(2)$-cover 
$$\mathcal{F}_1^*=\mathcal{F}_{10}\cup \{P_{11}, P_{12}, P_{13}, P_{14}\}\cup \{T_{21}, T_{22}, T_{23}, T_{24}\},$$ where $\mathcal{F}_{10}$ consists of signed circuits, $P_{11}$ and $P_{12}$ (resp., $P_{13}$ and $P_{14}$) are two positive (resp., negative) $x_0x_1$-paths, and $T_{11}$ and $T_{12}$ (resp., $T_{13}$ and $T_{14}$) is two tadpoles  at $x_0$ (resp., $x_1$).

Let $H=H(x_0, x_1)$ such that $G=\mathcal{P}(B_1, H)$. Choose $B_i=B_i(x_{i-1}, x_i)$, $i\in [2, s]$, such that  
$$H(x_1, x_0)=\mathcal{S}(B_2, B_3, \cdots, B_s)=B_2\cup B_3\cup \cdots \cup B_s$$
and $s$ is maximum with this property, where $x_1\in V(B_2)$ and $x_s=x_0\in V(B_s)$. Then $|\mathcal{B}_2(H)|\ge 1$; otherwise, by Claim \ref{le: easy observations}-(6), $H-L(H)$ is a positive or negative path with length $1$ or $2$, and thus one can easily find a signed circuit $6$-cover of $G$, a contradiction. Furthermore, $|\mathcal{B}_1(H)|+|\mathcal{B}_2(H)|\ge 2$ by Claim \ref{le: t<4}, and every $B_i\in \mathcal{B}_2(H)$ has a $\Psi_{x_{i-1}x_i}(2)$-cover by Claim \ref{le: a-b-covers and 2-2-covers}. Applying Lemma \ref{le: four paths}, we pick a signed subgraph $6$-cover $\mathcal{F}_2^*$ of $H$ as follows:
$$
\mathcal{F}_2^*=\mathcal{F}_{20}\cup 2\mathcal{B}_0(H)\cup \{P_{21}, P_{22}, P_{23}, P_{24}\}\cup \{T_{21}, T_{22}, T_{23}, T_{24}\},$$
 where $\mathcal{F}_{20}$ is a family of signed circuits, $P_{21}$ and $P_{22}$ (resp., $P_{23}$ and $P_{24}$) are two positive (resp., negative) $x_0x_1$-paths, $T_{21}$ and $T_{22}$ (resp., $T_{23}$ and $T_{24}$) are two tadpoles in $H$ at $x_0$ (resp., $x_1$)  whose unbalanced circuit is in the part in $\mathcal{B}_0(H)\cup \mathcal{B}_2(H)$ with maximum (resp., minimum) subscript.

Let  $U=\bigcap_{B\in \mathcal{B}_0(H)\cup \mathcal{B}_2(H)} V(B)$.  We first show $U\cap \{x_0, x_1\}=\emptyset$.  Otherwise $x_0\notin V(T_{23})\cap V(T_{24})$ and $x_1\notin V(T_{21})\cap V(T_{22})$.  Thus  the family
\begin{flalign*}
& \mathcal{F}_{10}  \cup \mathcal{F}_{20}  \cup \{P_{12}\cup P_{22}, P_{13}\cup P_{23}\}\cup \{T_{11}\cup T_{21}, T_{12}\cup T_{22}, T_{13}\cup T_{23}, T_{14}\cup T_{24}\}\\
 \cup  & 
\left\{
\begin{array}{ll}
\{P_{14}\cup P_{24}\}\cup \mathcal{C}_0 & \mbox{if } |\mathcal{B}_0(H)|\neq 1;\\
\{P_{11}\cup P_{24}\cup \mathcal{B}_0(H), P_{14}\cup P_{21}\cup \mathcal{B}_0(H)\} & \mbox{if $|\mathcal{B}_0(H)|=1$}
\end{array}
\right.
\end{flalign*}
is a  signed circuit $6$-cover of $G$, where $\mathcal{C}_0$ is a family of signed circuits of $P_{11}\cup P_{21}\cup \mathcal{B}_0(H)$ which covers $P_{11}\cup P_{21}$ once and $\mathcal{B}_0(H)$ twice, a contradcition. Hence $U\cap \{x_0, x_1\}\neq \emptyset$. 

WLOG, assume that $x_0\in U$. Then $\mathcal{B}_1(H)=\{B_2\}=\{x_1x_2\}$, $\mathcal{B}_2(H)=\{B_3\}$ and $\mathcal{B}_0(H)\in \{\emptyset, \{B_4\}\}$ since $|\mathcal{B}_1(H)|+|\mathcal{B}_2(H)|\ge 2$ and $|\mathcal{B}_2(H)|\ge 1$. Hence $G=B_4\cup \mathcal{P}(B_1\cup x_1x_2, B_3)$.

Note that $x_0z_1x_1z_2x_0\cup \{L_{z_1}, L_{z_2}\}$ has a $\Psi_{x_0x_1}(2)$-cover in which no tadpole at $x_1$ contains $x_0$. Since $B_1\cup x_1x_2$ is a piece of $G$ at $\{x_0,x_2\}$, by (1a) of Claim \ref{cl: h=2}, we have either $B_1=x_0x_1x_0$ or $x_0zx_1x_0\cup \{L_z\}$.  Since $B_3\cup x_2x_1$ is a piece of $G$ at $\{x_0, x_1\}$, it follows from Claim \ref{le: a-b-covers and 2-2-covers} and (1a) and (1b) of Claim \ref{cl: h=2} that either $B_3$ has a $\Psi_{x_0x_1}^*(2)$-cover or $B_3=B_3(x_0, x_2)\sim R_i(x, y)$ for some $i\in \{0, 2, 4, 5\}$. Therefore, 
by Lemma \ref{le: 2-sum-a-b}-(2), $G$ has a signed circuit $6$-cover, a contradiction. This completes the proof of Theorem \ref{TH: main}.

 \end{document}